\documentclass{svjour3}       % onecolumn (second format)
\smartqed  % flush right qed marks, e.g. at end of proof

\usepackage{amsmath, amssymb, color}
\usepackage[colorlinks=true,linkcolor=blue,urlcolor=blue]{hyperref}
\usepackage{graphicx}
\usepackage{caption}
\usepackage{mathtools}
\usepackage{enumerate}
\usepackage[all]{xypic}
\usepackage{verbatim}
\usepackage{tikz}
\usetikzlibrary{matrix}
\usetikzlibrary{arrows}
\usepackage[linesnumbered,lined,commentsnumbered,ruled]{algorithm2e}
\usepackage{caption}

\newtheorem{algo}[theorem]{Algorithm}

\newtheorem{notation}[theorem]{Notation}

\numberwithin{theorem}{section}

\newcommand{\RR}{\mathbb{R}}

\newcommand{\CC}{\mathbb{C} }

\newcommand{\NN}{\mathbb{N}}
\newcommand{\KK}{\mathbb{K}}
\newcommand{\FF}{\mathbb{F}}
\newcommand{\MM}{\mathcal{M}}
\newcommand{\pp}{\mathfrak{p}}
\newcommand{\Hom}{{\normalfont\text{Hom}}}
\newcommand{\Diff}{{\normalfont\text{Diff}}}
\makeatletter
\newcommand{\verbatimfont}[1]{\renewcommand{\verbatim@font}{\ttfamily#1}}
\makeatother

\begin{document}

\title{Primary Ideals and their Differential Equations} 

\author{Yairon Cid-Ruiz \and
	 		   Roser Homs  \and 
	 		   Bernd Sturmfels
}

\institute{Yairon Cid-Ruiz \at
	 Ghent University \\
	\email{Yairon.CidRuiz@UGent.be}             
	\and
	Roser Homs \at
	TU  M\"unchen \\
	\email{roser.homs@tum.de}
	\and
	Bernd Sturmfels \at
	MPI-MiS Leipzig and UC Berkeley \\
	\email{bernd@mis.mpg.de}
}

\maketitle

\vspace{2cm}

\begin{abstract} 
	An ideal in a polynomial ring encodes a system of linear partial differential equations with constant coefficients.
	Primary decomposition organizes the solutions to the PDE.
	This paper develops a novel structure theory for primary ideals in a polynomial ring.
	We characterize primary ideals in terms of PDE, punctual Hilbert schemes, 
	relative Weyl algebras, and the join construction.
	Solving the PDE described by a primary ideal amounts to computing
	Noetherian operators in the sense of Ehrenpreis and Palamodov.
	We develop new algorithms for this task, and
	we present efficient implementations.
\end{abstract}

\section{Introduction}
\label{sec1}

In his 1938 article  \cite{GROBNER_MATH_ANN} on the foundations of algebraic geometry,
Gr\"obner introduced differential operators 
to characterize membership in a polynomial ideal. He derived such characterizations for ideals that
are  prime or primary to a rational maximal ideal \cite[pages~174-178]{GROBNER_BOOK_AG_2}.
In a 1952 lecture  \cite[\S 1]{GROBNER_LIEGE} he
suggested that the same program can be carried out for~any 
primary ideal. 
Gr\"obner was particularly interested in algorithmic solutions to this problem.

Substantial contributions in this subject area were made by analysts.
In the 1960s, Ehrenpreis  \cite{EHRENPREIS} stated his \textit{Fundamental Principle} on solutions to
linear partial differential equations (PDE) with complex constant coefficients. A main step was the characterization of
primary ideals by differential operators. But, he incorrectly claimed that operators
with constant coefficients suffice. 
Using Example (\ref{exam:Palamodov}) below, Palamodov \cite{PALAMODOV} pointed out the error, and he gave a correct proof by introducing the representation by {\em Noetherian operators}.
Details on the Ehrenpreis-Palamodov Fundamental Principle can also be found in \cite{BJORK,HORMANDER}.

The ball returned to algebra in 1978 when Brumfiel
published the little-known paper~\cite{BRUMFIEL_DIFF_PRIM}.
In 1999, Oberst \cite{OBERST_NOETH_OPS} extended Palamodov's Noetherian operators to polynomial rings over arbitrary fields. 
In 2007, Damiano, Sabadini  and Struppa \cite{DAMIANO} gave a computational approach.
A general theory for Noetherian commutative rings was developed recently in \cite{NOETH_OPS}. 
Building on this, the present article develops a theory of primary ideals as envisioned by Gr\"obner.

We now introduce a running example that serves to illustrate our title and results.
The following prime ideal of codimension $c=2$ in $n=4$ variables is familiar to many algebraists:
\begin{equation}
\label{eq:twistedcubic1}
P \quad = \quad \langle\,
x_1^2-x_2 x_3,\, x_1 x_2 - x_3 x_4, x_2^2 - x_1 x_4 \,\rangle \quad \subset
\quad \CC[x_1,x_2,x_3,x_4].
\end{equation}
This ideal defines the (affine cone over the) {\em twisted cubic curve}  \cite{CUBIC_LITTLE}:
$$\,V(P) = \bigl\{ \,(s^2t, s t^2,s^3, t^3) \,:\, s,t \in \CC \,\bigr\}.$$ 
We
identify the polynomials in (\ref{eq:twistedcubic1}) with PDE with constant coefficients
by setting $x_i = \partial_{z_i}$.
Solving these PDE means describing all functions $\psi(z_1,z_2,z_3,z_4)$ with
\begin{equation}
\label{eq:twistedcubic2}
\frac{\partial^2 \psi}{\partial z_1^2} = \frac{\partial^2 \psi}{\partial z_2 \partial z_3}    \qquad {\rm and} \qquad
\frac{\partial^2 \psi}{\partial z_1 \partial z_2} = \frac{\partial^2 \psi}{\partial z_3 \partial z_4}   \qquad {\rm and} \qquad
\frac{\partial^2 \psi}{\partial z_2^2} = \frac{\partial^2 \psi}{\partial z_1 \partial z_4}.
\end{equation}
Analysis tells us that every solution comes from a measure $\mu$ on the $(s,t)$-plane:
\begin{equation}
\label{eq:twistedcubic4}
\psi(z_1,z_2,z_3,z_4) \,\,\,\, = \,\,\,
\int {\rm exp} \bigl(   z_1 s^2 t \,+\, z_2 s t^2 \,+\,z_3 s^3 \,+\, z_4 t^3
\bigr) \,\mu(s,t) \,{\rm d}s \,{\rm d}t .
\end{equation}
For instance, if $\mu$ is the Dirac measure at the point $(2,3)$ then 
this solution equals $\psi = {\rm exp}(  12 z_1 + 18 z_2 + 8 z_3 + 27 z_4)$.
Thus, the functions $\psi$ are simply an analytic encoding of the affine surface  $V(P) \subset \CC^4$.

The situation becomes interesting when we consider a non-reduced scheme structure
on our surface.   
Algebraically, this means replacing the prime $P$ by a $P$-primary ideal.
We  use differential operators to give 
compact representations of $P$-primary ideals $Q$. For instance,
\begin{equation}
\label{eq:twistedcubic6}
\begin{matrix}
Q \,\,=\, \, \bigl\{ \,f \in \CC[x_1,x_2,x_3,x_4]\,:\,
A_i \bullet f  \in P \,\,\,{\rm for} \,\,\, i=1,2,3\, \bigr\}, \smallskip \\ {\rm where} \quad
\,A_1 \,=\, 1\,,\;\, A_2\,=\, \partial_{x_1}
\,\,\,{\rm and} \,\,\,A_3 \,=\, \partial_{x_1}^2 \,-\, 2 \,x_2\,\partial_{x_2} . \quad 
\end{matrix}
\end{equation}
Here $\bullet$ means applying a differential operator to a function. A prime ideal is represented by 
just one Noetherian operator $A_1=1$.
We can encode (\ref{eq:twistedcubic6}) by the~ideal
\begin{align}
\label{eq:magic}
\begin{split}
	\bigl\langle 
	u_1^2 - u_2 u_3,  u_1 u_2 -u_3 u_4, 
	u_2^2-u_1 u_4,& 
	\\ x_1-u_1-y_1, x_2-u_2&-y_2, x_3-u_3, x_4-u_4, \,
	\underline{ y_1^3, \, y_2 + u_2 \,y_1^2}
	\bigr\rangle.
\end{split}  
\end{align}      
The role of the new variables $u_1,u_2,u_3,u_4,y_1,y_2$ is subtle. It will be explained
 in Section \ref{sec2}. 
 The minimal generators of $Q$ are obtained from~(\ref{eq:magic})
by eliminating $\{u_1,u_2,u_3,u_4,y_1,y_2\}$. We find that
$$ \begin{small} \begin{matrix} Q \,\, = & \,\;\; \bigl\langle\,
3 x_1^2 x_2^2-x_2^3 x_3-x_1^3 x_4-3 x_1 x_2 x_3 x_4+2 x_3^2 x_4^2\,,\,\, 
3 x_1^3 x_2 x_4-3 x_1 x_2^2 x_3 x_4-3 x_1^2 x_3 x_4^2+ \\ & 3 x_2 x_3^2 x_4^2 +2 x_2^3 -2 x_3 x_4^2\,,\,\,
3 x_2^4 x_3-6 x_1 x_2^2 x_3 x_4+3 x_1^2 x_3 x_4^2+x_2^3-x_3 x_4^2\,,\,\, 
4 x_1 x_2^3 x_3+ \\ & x_1^4 x_4-6 x_1^2 x_2 x_3 x_4 -3 x_2^2 x_3^2 x_4+4 x_1 x_3^2 x_4^2\,,\,\,
x_2^5-x_1 x_2^3 x_4-x_2^2 x_3 x_4^2+x_1 x_3 x_4^3\,,\,\, \\ & 
x_1 x_2^4-x_2^3 x_3 x_4-  x_1 x_2 x_3 x_4^2+x_3^2 x_4^3\,, 
x_1^4 x_2-x_2^3 x_3^2-2 x_1^3 x_3 x_4+2 x_1 x_2 x_3^2 x_4\,,\,\, \\ &
x_1^5-4 x_1^3 x_2 x_3+3 x_1 x_2^2 x_3^2+  2 x_1^2 x_3^2 x_4-2 x_2 x_3^3 x_4\,,
3 x_1^4 x_4^2-6 x_1^2 x_2 x_3 x_4^2+3 x_2^2 x_3^2 x_4^2 \\ & +4 x_2^4-4 x_2 x_3 x_4^2\,,\,\, 
x_2^3 x_3^2 x_4+x_1^3 x_3 x_4^2-  3 x_1 x_2 x_3^2 x_4^2+x_3^3 x_4^3  \!\! +x_1 x_2^3-x_1 x_3 x_4^2\,, \\ &
3 x_1^4 x_3 x_4-6 x_1^2 x_2 x_3^2 x_4+3 x_2^2 x_3^3 x_4+2 x_1^3 x_2
+6 x_1 x_2^2 x_3- 6 x_1^2 x_3 x_4-2 x_2 x_3^2 x_4\,,
\\ & 4 x_2^3 x_3^3+ 4 x_1^3 x_3^2 x_4-12 x_1 x_2 x_3^3 x_4  +4 x_3^4 x_4^2 
-x_1^4+6 x_1^2 x_2 x_3+3 x_2^2 x_3^2-8 x_1 x_3^2 x_4\,
\bigr\rangle.
\end{matrix} \end{small}
$$
As in (\ref{eq:twistedcubic1}) and (\ref{eq:twistedcubic2}), we
view $Q$ as  a system of PDE by setting
$x_i = \partial_{z_i}$. Its solutions are of the following form
for suitable measures $\mu_1,\mu_2,\mu_3$ on the $(s,t)$-plane $\CC^2$:
\begin{equation}
\label{eq:twistedcubic5}
\begin{matrix}
\hspace*{-8.3cm} \psi(z_1,z_2,z_3,z_4) \, = \\
\qquad\qquad \sum_{i=1}^3 \int \! B_i(z_1,z_2,s,t) \cdot
{\rm exp} \bigl(  z_1 s^2 t + z_2 s t^2 +z_3 s^3 +  z_4 t^3 \bigr)  \,\mu_i(s,t) \,{\rm d}s \,{\rm d}t ,
\smallskip \\ \qquad {\rm where}  \quad
B_1\,=\, 1\,,\;\, B_2 = z_1 \,\,\, {\rm and} \,\,\, B_3 \,=\, z_1^2  -  2 st^2 z_2 .
\end{matrix}
\end{equation}
Note that the primary ideal $Q$ has multiplicity $3$ over the prime ideal~$P$.

\smallskip

The title of this paper refers to two
ways of associating differential equations to a primary ideal in a polynomial ring.
First, we use PDE with polynomial coefficients, namely
Noetherian operators $A_i$ as in (\ref{eq:twistedcubic6}), 
to give a compact encoding of $Q$. Second, we can 
interpret $Q$ itself as a system of PDE with constant coefficients, with 
solutions represented by {\em Noetherian multipliers} $B_i$ as in (\ref{eq:twistedcubic5}).
The dual roles played by the $A_i$ and $B_i$ is one of our main themes.

\smallskip

This paper is organized as follows.
In Section \ref{sec2} we state our main result, namely the
  characterization of primary ideals in terms of punctual Hilbert schemes and Weyl-Noether modules.
The former offers a parametrization of all $P$-primary ideals of a given multiplicity, and the latter establishes the links to differential equations.
In Section \ref{sec3} we turn to the Ehrenpreis-Palamodov
Fundamental Principle. 
We present a self-contained proof of the algebraic part. 
In Sections \ref{sec4} and \ref{sec6} we prove the results stated in Section \ref{sec2}.
Section \ref{sec5} reviews differential operators in commutative algebra
and supplies tools for our proofs.
In Section \ref{sec7} we study the join construction for primary ideals, which offers a new perspective on ideals that are similar to symbolic powers.
In Section~\ref{sec8} we introduce  algorithms for computing Noetherian operators
and hence for solving linear PDE with constant coefficients.

\section{Characterizing Primary Ideals}
\label{sec2}

Irreducible varieties and their prime ideals are the basic building blocks in algebraic geometry.
Solving systems of polynomial equations
means extracting the associated primes from the system, and to
subsequently study their irreducible varieties. However, if the given ideal
is not radical then we seek the primary decomposition and not just the 
associated primes. We wish to gain a precise understanding of the
primary ideals that make up the given scheme.

We furnish a representation theorem for primary ideals in a polynomial ring,
extending the familiar case of zero-dimensional ideals (Macaulay's inverse system \cite{Grobner37}).
Fix a field $\KK$ of characteristic zero and 
a prime ideal $P$ of codimension $c$ in the polynomial ring   $R = \KK[x_1,\ldots,x_n]$.
We write $\FF$ for the field of fractions of the integral domain $R/P$. 

		The main contribution of this paper is the theorem below. 
		We parametrize all $P$-primary ideals of a given multiplicity over $P$  via three different but closely related objects.
		Of particular interest is our characterization of $P$-primary ideals as points in a \emph{punctual Hilbert scheme} --- an object of fundamental importance in algebraic geometry. 
		This tells us that the space of all $P$-primary ideals
		has a rich geometrical structure.
		While Theorem~\ref{thm:main} is a theoretical contribution in commutative algebra,
		it leads to efficient algorithms for going back and forth between primary ideals 
		and Noetherian operators, to be discussed in Section~\ref{sec8}.

\begin{theorem} \label{thm:main}
	The following sets of objects are in a natural bijective correspondence:
	\begin{enumerate}[(a)]
		\item $P$-primary ideals $Q$ in $R$ of multiplicity $m$ over $P$, 
		\item points in the punctual Hilbert scheme $\,{\rm Hilb}^m(\FF[[y_1,\ldots,y_c]])$,
		\item $m$-dimensional $\FF$-subspaces of $\,\FF[z_1,\ldots,z_c]$ that are closed under differentiation,
		\item  $m$-dimensional $\FF$-subspaces of
		the Weyl-Noether module $ \FF \,\otimes_R  \,D_{n,c}$ that are $R$-bimodules. 
	\end{enumerate}
	Moreover, any basis of the $\FF$-subspace in (d) lifts to Noetherian operators~$A_1,\ldots,A_m$ in the relative Weyl algebra $D_{n,c}$ that represent the ideal $Q$ in (a) as in~(\ref{eq:fromAtoQ}).
\end{theorem}

The purpose of this section is to define and explain all the 
concepts in Theorem \ref{thm:main}. Our aim is to state
the promised bijections as explicitly as possible. 
The proof of Theorem \ref{thm:main}
will be divided into smaller pieces and given in Sections \ref{sec4} and \ref{sec6}.
The encoding of $Q$ by
Noetherian operators $A_i$ will be explained in Section~\ref{sec3}.	
We already saw an example in (\ref{eq:twistedcubic6}).
The Weyl-Noether module in part (d)
is our stage for the PDE that portray primary ideals.

\smallskip

We begin by returning to Gr\"obner. His 1937
article \cite{Grobner37}  interpreted Macaulay's inverse system as solutions
to linear PDE. 
He considered the special case when $P = \langle x_1,\ldots,x_n \rangle $ is the homogeneous maximal ideal, so we have $c=n$ and $\FF = \KK$.
The geometric intuition invoked in \cite[\S 1]{GROBNER_LIEGE}
is captured by the punctual Hilbert scheme $\,{\rm Hilb}^m(\KK[[y_1,\ldots,y_n]])$,
whose points are precisely the $P$-primary ideals of colength~$m$.
This zero-dimensional case is familiar to most commutative algebraists,
especially the readers of \cite{MOURRAIN_DUALITY}. Here,
parts (c) and (d) of Theorem~\ref{thm:main} refer to the $m$-dimensional $\KK$-vector 
space of polynomial solutions to the~PDE.

The general case of higher-dimensional primary ideals $Q$
was of great interest to Gr\"obner. In his 1952 Li\`ege lecture \cite{GROBNER_LIEGE}, he 
points to Severi \cite{Severi}, and he writes:
{\em En ce sense la vari\'et\'e alg\'ebrique correspondant \`a un id\'eal
	primaire $Q$ pour l'id\'eal premier $P$ consiste en les points ordinaires de la
	vari\'et\'e $\,V(P)$ et en certain nombre $m$ de points
	infiniment voisins, c'est-\`a-dire dans $m$ conditions diff\'erentielles 
	ajout\'ees \`a chaque point de la vari\'et\'e $\,V(P)$. Le nombre $m$ de ces
	conditions diff\'erentielles est \'egal \`a la longueur de l'id\'eal primaire $Q$.}

But Gr\"obner was never able to complete the program himself, in spite of
the optimism he still expressed in his 1970 textbook \cite{GROBNER_BOOK_AG_2}.
After the detailed treatment of Macaulay's inverse systems for zero-dimensional ideals, he proclaims:
{\em Es d\"urfte auch nicht schwer sein den oben angegebenen Formalismus
	auf mehrdimensionale Prim\"arideale auszudehnen}
\cite[page 178]{GROBNER_BOOK_AG_2}.

The issue was finally resolved by the theory of
Ehrenpreis-Palamodov \cite{EHRENPREIS,PALAMODOV},
presented in Section \ref{sec3},
and the subsequent developments \cite{BRUMFIEL_DIFF_PRIM,NOETH_OPS,DAMIANO,OBERST_NOETH_OPS}
we discussed in the Introduction.
We regard  our Theorem \ref{thm:main} as a rather definitive result on primary ideals in $R$.
It captures the geometric spirit of Gr\"obner and Severi, as it explains
their ``infinitely near points'' in the language of modern algebraic geometry, namely
using Hilbert schemes. This opens up the possibility 
of developing numerical methods for primary ideals and their differential equations,
by linking our results to current advances in
numerical algebraic geometry \cite{KRONE,LEYKIN}.

Two essential ingredients in Theorem \ref{thm:main}
are the function field $\FF$  and the Weyl-Noether module
$\, \FF  \otimes_R D_{n,c}$.
We start our technical discussion with some insights into these objects.
	Since $\dim(R/P) = n-c$, after permuting variables, we assume that $\{x_{c+1},\ldots,x_n\}$ is a \emph{maximal independent set of variables modulo $P$}. 	This means that
		 $\KK[x_{c+1},\ldots,x_n] \cap P = \{0\}$; see
	\cite{KREDEL_INDEP_SET}.
	 Expressed in combinatorics language, 
	 our assumption says that 
	 $\{x_{c+1},\ldots,x_n\}$ 
	 is a basis of the {\em algebraic matroid} given by $P$. This implies that
	$\FF = \text{Quot}(R/P)$ is algebraic over the field
	$\KK(x_{c+1},\ldots,x_n)$, which is a purely transcendental extension of the ground field~$\KK$. 

Throughout the existing literature on the construction  of Noetherian operators,
the authors have relied on the process of Noether normalization  (see, e.g., \cite[Chapter 8]{BJORK}, \cite{OBERST_NOETH_OPS}).  They assumed that
	  the quotient ring $R/P$ is a finitely generated module over the polynomial subring $\KK[x_{c+1},\ldots,x_n]$. 
	In our view, this hypothesis is  highly undesirable because it requires
a linear change of coordinates. Changing coordinates can drastically increase the size of the
polynomials and PDE one computes with.
	We here get rid of that hypothesis. {\em We do not use Noether normalization}.
	Instead we fix any
	 maximal independent set of variables modulo~$P$.

Clear notation is very important for this article.
This is why multiple letters $x,y,z,u$ are used to denote variables
and differential operators. Elements in $\FF$ are represented as fractions of polynomials in 
$\KK[u_1,\dots,u_n]$, where $u_i$ denotes the residue class of $x_i$ modulo $P$. 
Whenever the number $n$ of variables is clear from the context, we use the multi-index notation $\mathbf{u}^\alpha=u_1^{\alpha_1}\cdots u_n^{\alpha_n}$. 
Elements $a(\mathbf{u})/b(\mathbf{u})$ of the field $\FF$ are represented by
taking $a(\mathbf{u})$ and $b(\mathbf{u})$ coprime and in normal~form
with respect to a Gr\"obner basis of $P$.
Arithmetic in $\FF$ is performed via this Gr\"obner basis.
The $R$-module structure of $\FF$ is given by $\,\mathbf{x}^\alpha\cdot a(\mathbf{u})/b(\mathbf{u})=\mathbf{u}^\alpha a(\mathbf{u})/b(\mathbf{u})$. 
Alternatively, from the perspective of numerical algebraic geometry, a better approach to arithmetic in $\FF$ is to work with generic points, obtained by realizing $R/P$ as a subring of a suitable field of functions on $V(P)$.
In our running example, that suitable field could be $\,\KK(s/t,t^3)$. It contains
$R/P$ as the subring $\KK[s^2 t, st^2, s^3, t^3]$.

The {\em relative Weyl algebra}   $D_{n,c} = \KK \langle x_1,\ldots,x_n, \partial_{x_1},\dots,\partial_{x_c}\rangle$
is the $\KK$-algebra on  $n{+}c$ generators 
$x_1,\ldots,x_n,\partial_{x_1},\ldots,\partial_{x_c}$ that commute  
except for $\partial_{x_i}x_i=x_i\partial_{x_i}{+}1$. 
This is a subalgebra of the  Weyl algebra, so $D_{n,c}$ is non-commutative. 
Its elements are linear differential operators with polynomial coefficients,
where derivatives occur with respect to the first $c$ variables. 
The following set is a $\KK$-basis of $D_{n,c}$:
$$
\left\lbrace x_1^{\alpha_1}\cdots x_n^{\alpha_n}{\partial_{x_1}}^{\!\!\!\!\beta_1}\cdots
{\partial_{x_c}}^{\!\!\!\!\beta_c}:(\alpha,\beta)\in\mathbb{N}^n\times\mathbb{N}^c\right\rbrace
$$ 

The {\em Weyl-Noether module} of the affine variety $V(P)$ is the tensor product
\begin{equation}
\label{eq:relativeweyl}
\FF\,\otimes_R\,D_{n,c} \,\, = \,\, \FF \,\otimes_R \,
R \langle \partial_{x_1},\ldots,\partial_{x_c}\rangle.
\end{equation}
Since $\FF$ is the field of fractions of the integral domain $R/P$,  it is clearly 
an $R$-module. Note that the relative Weyl algebra
$D_{n,c} = R\langle \partial_{x_1},\dots,\partial_{x_c}\rangle$
is non-commutative, and it has two distinct  $R$-module structures:
it is a left $R$-module and it is a right $R$-module.
In the tensor product (\ref{eq:relativeweyl}),
for convenience of notation, we mean the left $R$-module structure
on $D_{n,c}$. Later, in Remark~\ref{rem_isom_restrict_Weyl_mod}, we shall 
give an intrinsic description of $\FF \otimes_R D_{n,c}$ with differential operators.

The Weyl-Noether module (\ref{eq:relativeweyl}) 
has both right and left $R$-module structures.
The action by $R$ on the left is easy to write using the standard $\KK$-basis above:
\begin{equation}
\label{eq:leftaction}
\mathbf{x}^\alpha \cdot \biggl( \frac{a(\mathbf{u})}{b(\mathbf{u})} \,\otimes_R\, \partial_{\mathbf{x}}^\beta\ \biggr) \,\,\, = \,\,\,
\frac{\mathbf{u}^{\alpha} a(\mathbf{u})}{b(\mathbf{u})} \,\otimes_R\, \partial_{\mathbf{x}}^\beta.
\end{equation}
For the action on the right we need the commutation identities in the Weyl algebra:
$$\partial_{\mathbf{x}}^\beta \mathbf{x}^\alpha\,\,=\,\,\sum_{\gamma, \delta} \lambda_{\gamma,\delta}
\,  \mathbf{x}^\gamma \partial_\mathbf{x}^\delta.$$
Here $\lambda_{\gamma,\delta}$ are the positive integers in \cite[Problem 4]{SatStu}.
With this, the right action~is
\begin{equation}
\label{eq:rightaction}
\biggl( \frac{a(\mathbf{u})}{b(\mathbf{u})}  \,\otimes_R \, \partial_{\mathbf{x}}^\beta \biggr)
\cdot \mathbf{x}^\alpha  \, \,\, = \,\,\, \frac{a(\mathbf{u})}{b(\mathbf{u})}
\, \otimes_R \, \partial_{\mathbf{x}}^\beta \mathbf{x}^\alpha \,\,\,=\,\,\,
\sum_{\gamma,\delta} \lambda_{\gamma,\delta} \, \frac{\mathbf{u}^{\gamma}a(\mathbf{u})}{b(\mathbf{u})}  \, \otimes_R \, 
\partial_\mathbf{x}^\delta.
\end{equation}
This means that the $R$-bimodule condition in Theorem~\ref{thm:main}~(d)
is very strong.

From the action (\ref{eq:leftaction}) we deduce that  $\,\FF\,\otimes_R\,D_{n,c}\,$ is a left $\FF$-vector space with basis 
$\,\left\lbrace 1 \otimes_R \partial_{\mathbf{x}}^\beta: \beta\in\mathbb{N}^c\right\rbrace$,
so we could also write $\FF\langle \partial_{x_1},\dots,\partial_{x_c}\rangle$
for (\ref{eq:relativeweyl}).
However, we prefer the previous notation because it highlights
that  there are two distinct structures. The  Weyl-Noether module
is a left $\FF$-vector space via 
(\ref{eq:leftaction}) and it is a right $R$-module via (\ref{eq:rightaction}). 
It is not a right $\FF$-vector space because the right $R$-action
is not compatible with passing to $R/P$:

\begin{example} Fix the maximal ideal $P=\langle x_1, \ldots,x_n\rangle$
	so that  $\FF=\KK$ and $c=n$. 
	Since $\overline{x_j} = 0 \in R/P$, we have $\,x_j \cdot \left(1 \otimes_R \partial_{x_j}\right) = 0 \,$  and hence $\,\left(1 \otimes_R \partial_{x_j}\right) \cdot x_j = 1 \otimes_R 1$ holds~in $\,\FF\,\otimes_R\,D_{n,c}$.
	This shows that there is no right $\FF$-action on the Weyl-Noether module $\,\FF\,\otimes_R\,D_{n,c}$.
\end{example}

We now come to our parameter space in part (b), the punctual Hilbert~scheme 
\begin{equation}
\label{eq:hilbertscheme}
{\rm Hilb}^m \bigl( \,\FF[[y_1,\ldots,y_c]] \,\bigr). 
\end{equation}
This is a quasiprojective scheme over the function field $\FF$.
Its classical points are   ideals of colength $m$ in the local ring $\FF[[y_1,\ldots,y_c]]$.
By Cohen's Structure Theorem,  this ring is the  completion of  $R_P$, the localization of $R$ at the prime $P$.
To connect parts (a) and~(b), we recall that the
multiplicity $m$ of a primary ideal $Q$ over its prime $P = \sqrt{Q}$ 
is the length of the artinian local ring $\,R_P/Q R_P$.
In symbols, using the command {\tt degree}  in  {\tt Macaulay2}~\cite{MAC2},  we have
$$ m \,\,=\,\, {\rm length}\bigl( R_P/QR_P \bigr) \,\,=\,\, \frac{{\tt degree}(Q)}{{\tt degree}(P)}.$$

The punctual Hilbert scheme (\ref{eq:hilbertscheme}) is familiar to algebraic geometers, 
but its structure is very complicated when $c \geq 3$.
We refer to Iarrobino's article \cite{IARROBINO_HILB} as a point of entry,
and to Jelisiejew's recent papers \cite{JEL1,JEL2} for the intriguing pathologies in this subject.
While the punctual Hilbert scheme is trivial for $c=1$, Brian\c{c}on \cite{Bri77} undertook a detailed study for $c=2$.
He showed that $\,{\rm Hilb}^m \bigl( \FF[[y_1,y_2]] \bigr)\,$ is smooth and irreducible of dimension $m-1$.
A dense subset is given by the $(m-1)$-dimensional family of $\langle y_1,y_2 \rangle$-primary ideals of the form
\begin{equation}
\label{eq:HSfamily}
\qquad
\bigl\langle  \,\, y_1^m\,,\, \,y_2 \,+\, a_1 y_1 + a_2 y_1^2 +  \cdots +  a_{m-1} y_1^{m-1} \,
\bigr\rangle, \quad
{\rm where}\,\,a_1,a_2,\ldots, a_{m-1} \in \FF.
\end{equation}
For instance, for $m=3$, the Hilbert scheme (\ref{eq:hilbertscheme}) is a surface over $\FF$.
Each of its points encodes a scheme structure of multiplicity $3$ on the variety $V(P)$. 
This is the generic point on $V(P)$ together with two ``infinitely near points'', in the language of Gr\"obner and Severi.
To see that the family (\ref{eq:HSfamily}) is a proper subset of
$\,{\rm Hilb}^m \bigl( \FF[[y_1,y_2]] \bigr)$, we consider the points
$$ \qquad \langle\,  y_1^3 \,, \,y_2 + \epsilon^{-1} y_1^2 \,\rangle\,\,  =\,\,
\langle \,y_1^2 + \epsilon y_2\, , \,y_1y_2\, ,\,y_2^2 \,\rangle \quad \in \,\,\,
{\rm Hilb}^3 \bigl( \FF[[y_1,y_2]] \bigr). 
$$
For $\epsilon \in \FF \backslash \{0\}$, this $\langle y_1,y_2 \rangle $-primary
ideal is in the family (\ref{eq:HSfamily}); for $\epsilon = 0$ it is not.

\begin{remark}
	In the zero-dimensional case, 
	when $P = \langle x_1,\ldots,x_n \rangle$, the
	correspondences in Theorem \ref{thm:main} 
	are well-known since the 1930's.
	Wolfgang Gr\"obner tells us:
	{\em Die noch verbleibende Aufgabe, die Integrale eines Prim\"arideals aus denjenigen f\"ur das zugeh\"orige  Primideal abzuleiten, wollen wir hier wenigstens f\"ur null-dimensionale Prim\"arideale allgemein l\"osen}
	\cite[page 272]{Grobner39}.
	In our current understanding, the
	$P$-primary ideals are points in ${\rm Hilb}^m(\KK[[y_1,\dots,y_n]])$, subspaces closed by differentiation are Macaulay's inverse systems, and 
	these account for polynomial solutions to linear PDE 
	with constant coefficients \cite{MOURRAIN_DUALITY,STURMFELS_SOLVING}.
\end{remark}

\begin{remark} Some subtleties regarding the punctual Hilbert scheme are worth mentioning. Finite subschemes of length $m$ supported at the origin of $\FF[[y_1,...,y_c]]$ can be parametrized by three different objects:

\smallskip
\noindent
${\rm Hilb}^m(\FF[[y_1,...,y_c]])$, ${\rm Hilb}^m(\FF[[y_1,...,y_c]]/(y_1,...,y_c)^m)$ and
the locus in ${\rm Hilb}^m(\mathbb{A}^c)$ of subschemes that are  supported at the origin. 

\smallskip
These schemes have the same set of closed points but differ in their properties. In Theorem~\ref{thm:main} we 
parametrize zero-dimensional ideals in $\FF[[y_1,…,y_c]]$ of colength $m$, 
so we are only interested in the closed points. 
The ambiguities in defining the punctual Hilbert scheme do not affect our results.
\end{remark}

\medskip

The idea behind Theorem~\ref{thm:main} is to reduce the study of
arbitrary primary ideals in  $R = \KK[x_1,\ldots,x_n]$ to a
zero-dimensional setting over the function field $\FF$.
Recall that coordinates were chosen so that
$\FF=\text{Quot}(R/P)$ is algebraic over $\KK(x_{c+1},\ldots,x_n)$.
Consider the  inclusion
\begin{equation}
\label{eq_map_gamma}
\gamma:R \hookrightarrow \FF[y_1,\dots,y_c]\, , \qquad
\begin{matrix}
x_i  &\mapsto &  y_i+u_i, & \!\!\!\!\! \mbox{ for }1\leq i\leq c,\\
x_j & \mapsto  & u_j,& \quad \mbox{ for }c+1\leq j\leq n,
\end{matrix}
\end{equation}
where $u_i$ denotes the class of $x_i$ in $\FF$, for $1\leq i\leq n$.
From this we get an explicit 
correspondence between the objects in parts (a) and (b) of
Theorem \ref{thm:main}:
\begin{equation}
\label{eq:corr12}
\begin{array}{ccc}
\left\lbrace\begin{array}{c}
\mbox{$P$-primary ideals of $R$}\\
\mbox{with multiplicity $m$ over $P$}
\end{array}\right\rbrace
& \longleftrightarrow &
\left\lbrace\begin{array}{c}
\mbox{points in }{\rm Hilb}^m(\FF[[y_1,\ldots,y_c]])\\
\end{array}\right\rbrace\\
Q & \longrightarrow & I=\langle y_1,\dots,y_c\rangle^m+\gamma(
Q)\FF[y_1,\dots,y_c]\\
Q\,=\,\gamma^{-1}(I) & \longleftarrow & I.
\end{array}
\end{equation}

\begin{example}\label{Bij:1,2} 
	Fix $P $ and $Q$ as in the Introduction, with $n=4$, $m=3$, $c=2$,  where
	$\FF$ is algebraic over $\CC(x_3,x_4)$. The primary ideal $Q$
	corresponds to a point in ${\rm Hilb}^3(\FF[[y_1,y_2]])$. See \cite[Section IV.2]{Bri77} for a detailed description of points in the Hilbert scheme of degree $3$ in two variables. 
	The bijection 
	in (\ref{eq:corr12}) gives the following point in the punctual Hilbert scheme: 
	\begin{equation}
	\label{eq:punctual2}
	I \,\,=\,\,\langle y_2^2,y_1y_2,y_1^2+u_2^{\,-1}y_2\rangle
	\,\,\,\subset \,\,\FF[[y_1,y_2]].
	\end{equation}
	Note that this ideal is also generated by $y_1^3$ and $y_2+u_2y_1^2$, as in (\ref{eq:magic}).
\end{example}

\smallskip

The bijection between (b) and (c) is Macaulay's duality between ideals of dimension zero in a power series ring and finite-dimensional subspaces in a polynomial ring that are closed under differentiation.  
To interpret polynomials in $I$ as PDE, we replace $y_i$ by $\partial_{z_i}$.
So, by slight abuse of notation, we shall  write $\FF[[y_1,\ldots,y_c]]$ and $\FF[[\partial_{z_1}, \ldots, \partial_{z_c}]]$ 
interchangeably.
With this, the {\em inverse system} of a zero-dimensional ideal $I$ in the local ring $\FF[[\partial_{z_1},\dots,\partial_{z_c}]]$, denoted by $I^\perp$, is the $\FF$-vector space of solutions $\,\left\lbrace F\in \FF[z_1,\ldots,z_c]: f\bullet F=0 \mbox{ for all }f\in I\right\rbrace$.

Inverse systems give an explicit bijection between  (b) and (c) in Theorem~\ref{thm:main}:
\begin{equation}
\label{eq:HilbBijection}
\begin{array}{ccc}
\left\lbrace\begin{array}{c}
\mbox{points in }{\rm Hilb}^m\left(\FF[[\partial_{z_1},\ldots,\partial_{z_c}]]\right)\\
\end{array}\right\rbrace
& \longleftrightarrow &
\left\lbrace\begin{array}{c}
\mbox{$m$-dimensional $\FF$-subspaces} \\ 
\mbox{of $ \FF[z_1,\dots,z_c]$} \\ \mbox{closed under differentiation}
\end{array}\right\rbrace\\
I & \longrightarrow & V=I^\perp\\
I={\rm Ann}_{\FF[[\partial_{z_1},\dots,\partial_{z_c}]]}(V) & \longleftarrow & V.\\
\end{array}
\end{equation}

\begin{example} Setting $y_i=\partial_{z_i}$, the ideal in Example~\ref{Bij:1,2} is
	$I=\langle\partial_{z_2}^2,\partial_{z_1}\partial_{z_2},\partial_{z_1}^2+u_2^{\, -1}\partial_{z_2}\rangle\subset \FF[[\partial_{z_1},\partial_{z_2}]]$.
	Note that $\,z_1^2-2u_2z_2\,$ belongs to the inverse system $I^\perp$
	because this polynomial is annihilated by all operators in $I$.
	Applying the differential operators $\partial_{z_1}$ and $\partial_{z_1}^2$ to $B_3=z_1^2-2u_2z_2$ we obtain an $\FF$-basis of the inverse system: $B_1=1$, $B_2=z_1$ and $B_3$. Moreover, $I^\perp$ is generated by $B_3$ as an $\FF[[\partial_{z_1},\partial_{z_2}]]$-module.
	Hence $I$ is a Gorenstein ideal.
\end{example}

The correspondence between items (c) and (d) in Theorem~\ref{thm:main} links 
generators of the inverse system of $I$ with Noetherian operators for $Q$. 
These will be discussed in depth in Section~\ref{sec3}. 
Suppose we are given an $\FF$-basis $\{B_1,\ldots,B_m\}$
of the inverse system $I^\perp$ in~(c).
After clearing denominators, we can write $B_i(\mathbf{u},\mathbf{z})=\sum_{\vert\alpha\vert\leq m}\lambda_\alpha(\mathbf{u})\mathbf{z}^\alpha$ 
where $\lambda_\alpha(\mathbf{u})$ is a polynomial in $R$ 
that represents a residue class modulo $P$.
We now replace the unknown $z_i$ in these polynomials with the
differential operator $\partial_{x_i}$. This gives the Noetherian operators 
\begin{equation}
\label{eq:resultingDO} \qquad
A_i(\mathbf{x},\partial_{x_1},\dots,\partial_{x_c})\,\,\,=\,\,\,\sum_{\vert\alpha\vert\leq m} \lambda_\alpha(\mathbf{x}) \partial_{x_1}^{\alpha_1}\cdots \partial_{x_c}^{\alpha_c} 
\qquad {\rm for} \,\,\, i =1,\ldots,m.
\end{equation}
The map from the $B_i$ to the $A_i$ is invertible, giving the bijection between~(c)~and~(d).

\begin{example}\label{ex:noeth}
	Consider the ideal $Q$ in (\ref{eq:twistedcubic6})
	and $I$ in (\ref{eq:punctual2}).
	From the generators $B_1(u,z)=1$, $B_2(u,z)=z_1$ and $B_3(u,z)=z_1^2-2u_2z_2$ of the inverse system $I^\perp$ in $\FF[z_1,z_2]$, we obtain the three Noetherian operators $A_1=1$, $A_2=\partial_{x_1}$ and $A_3=\partial_{x_1}^2-2x_2\partial_{x_2}$
	that encode $Q$.
	Note that $A_3$ alone does not determine $Q$, although $B_3$ is enough to generate the inverse system.
\end{example}

\section{Solving PDE via Noetherian Multipliers}
\label{sec3}

 In this section we delve into the description of primary ideals in terms of Noetherian operators, and
 we explain the connection with solving systems of linear PDE via the Fundamental Principle of Ehrenpreis \cite{EHRENPREIS} and Palamodov \cite{PALAMODOV}. In particular, we show how Theorem~\ref{thm:main}
 leads to an integral representation of the solutions. The kernels are given by the Noetherian multipliers
 $B_i$.
 We saw a first example of this in (\ref{eq:twistedcubic5}).
For analytic aspects of the Ehrenpreis-Palamodov Theorem we refer to \cite{EHRENPREIS,PALAMODOV} and to the books by  Bj\"ork \cite{BJORK} and H\"ormander~\cite{HORMANDER}.

Our point of departure is a prime ideal $P$ in the polynomial ring $R = \KK[x_1,\ldots,$ $x_n]$.
We are interested in $P$-primary ideals.
Later on we shall interpret these ideals as systems of linear PDE, by replacing each variable
$x_i$ by a differential operator $ \partial_{z_i} = \partial / \partial z_i$. 
First, however, we take a different path, aimed at turning part (d) in Theorem \ref{thm:main} into an algorithm.

After choosing a maximal independent set and possibly permuting the variables, 
the field $\FF=\text{Quot}(R/P)$ is an algebraic extension of $\KK(x_{c+1},\ldots,x_n)$,
	where $c={\rm codim}(P)$.
The relative Weyl algebra $D_{n,c} = \KK \langle x_1,\ldots,x_n,\partial_{x_1},\ldots,\partial_{x_c} \rangle$
consists of all linear differential operators with polynomial coefficients, where only derivatives
for the first $c$ variables appear.
Every operator $A = A(\mathbf{x}, \partial_\mathbf{x})$ in  $D_{n,c}$ is a unique $\KK$-linear combination  of {\em standard monomials}  $\,\mathbf{x}^\alpha \partial_\mathbf{x}^\beta = x_1^{\alpha_1} \cdots x_n^{\alpha_n} \partial_{x_1}^{\beta_1} \cdots \partial_{x_c}^{\beta_c}$, where $\alpha \in \mathbb{N}^n$, $\beta \in \NN^c$.
We write $A \bullet f$ for the natural action of $D_{n,c}$ on
polynomials $f  \in R$, which is defined by
$$   x_i \bullet f = x_i \cdot f \quad {\rm and} \quad \partial_{x_i} \bullet f = \partial f / \partial x_i . $$

Consider elements $A_1,\ldots,A_m$ in the relative Weyl algebra $D_{n,c}$. These specify
\begin{equation}
\label{eq:fromAtoQ}
Q \,\,=\,\, \big\{\, f \in R \,: A_l \bullet f \in P \;\text{ for } \,
l = 1,2,\ldots,m	\,\big\}.
\end{equation}
The set $Q$ is a $\KK$-vector space. But, in general, the subspace $Q$ is not an ideal in~$R$.

\begin{example} \label{ex:leftright}
	Fix $n=m=2$, $P = \langle x_1,x_2 \rangle$ and $A_1 = \partial_{x_1}$.
	If  $A_2 = \partial_{x_2}$ then
	$Q $ is the space of polynomials $f$ in $\KK[x_1,x_2]$
	such that $x_1$ and $x_2$ do not appear in the expansion of~$f$.
	That space is not an ideal. 
	But, if $A_2 = 1$ then the formula
	(\ref{eq:fromAtoQ}) gives the ideal $\,Q = \langle x_1^2, x_2 \rangle$.
\end{example}

\begin{remark}
	\label{rem:contains_power}
	The space $Q$ always contains a power of $P$. Namely,
	by the product rule of calculus, if $k$
	is the maximal order among the operators $A_i$
	then $P^{k+1} \subseteq Q$. 
	\end{remark}

We next present a necessary and sufficient condition
for $m$ operators in $D_{n,c}$ to specify a primary ideal via (\ref{eq:fromAtoQ}).
We abbreviate $S= \KK(x_{c+1},\ldots,x_n)[x_1,\ldots,x_c]$.
The point in (\ref{eq:leftequalsright}) below
is that the relative Weyl algebra $D_{n,c}$ is both a left $R$-module and a right $R$-module.

\begin{theorem} \label{thm:leftequalsright}
	The space $Q$ is a $P$-primary ideal in the 
	polynomial ring $R$ if and only~if
	\begin{equation}
	\label{eq:leftequalsright}
\!\!	 A_i \cdot x_j\,\in \,
	S\cdot \{A_1,\ldots,A_m\} + P \cdot S\langle\partial_{x_1},\dots,\partial_{x_c}\rangle
	\quad \text{for}\, \, i =1,\ldots,m \,\text{and} \,j =1,\ldots,n.
	\end{equation}
\end{theorem}

In  Example \ref{ex:leftright} with $\{A_1,A_2\} = \{\partial_{x_1},\partial_{x_2} \}$
we have $ R = S$. Here $Q$ is not an ideal, and (\ref{eq:leftequalsright}) fails indeed for  $i=j=1$.
To see this, we note 
$  \partial_{x_1} x_1 \not\in R\cdot \{\partial_{x_1},\partial_{x_2}\} + \langle x_1,x_2 \rangle \cdot R\langle \partial_{x_1},\partial_{x_2}\rangle$.
It would be desirable to turn the criterion in Theorem \ref{thm:leftequalsright} into a practical algorithm.

\begin{proof}[Theorem \ref{thm:leftequalsright}]
	Suppose  (\ref{eq:leftequalsright}) holds and let $f \in Q$.
	By hypothesis, there exist $\,h_1,\ldots,h_m \in S$ such that
	$\,A_i x_j  \,=\, \sum_{k=1}^m h_k \,A_k \,$ modulo $P \cdot S\langle\partial_{x_1},\dots,\partial_{x_c}\rangle$.
	Since $A_k \bullet f \in P$, we see that
	$A_i \bullet (x_j f) = (A_i \,x_j) \bullet f$ lies in $P$ for all $i,j$.
	Hence $x_j f \in Q$. So, $Q$ is an ideal.
	
The following direct argument shows that $Q$ is $P$-primary.
	Let $f,g \in R$ such that $f \cdot g \in Q$ and $g \not\in Q$.
	We claim that $f \in P$. 
	We select an operator $A$ of minimal order among those inside $S\cdot \{A_1,\ldots,A_m\} + P \cdot S\langle\partial_{x_1},\dots,\partial_{x_c}\rangle$ that satisfy $A \bullet g \not\in PS$.
	The element $\,A \bullet (fg) \,=\, f \cdot (A \bullet g) \,+\, (A f - f A) \bullet g \,$ lies in $PS$. 
	The commutator $A f - f A$
	is a differential operator of order smaller than that of $A$.
	By (\ref{eq:leftequalsright}), it is inside $S\cdot \{A_1,\ldots,A_m\} + P \cdot S\langle\partial_{x_1},\dots,\partial_{x_c}\rangle$. 
	This ensures that $(A f - f A) \bullet g$ is in $PS$.
	We conclude that $f \cdot (A \bullet g) \in PS$. 
	But, we know that $A \bullet g$ is not in $PS$, and hence $f$ is in the prime ideal $P$.
	Remark~\ref{rem:contains_power} ensures that $\sqrt{Q}$ contains $ P$.
	Our argument shows that $Q$ is primary with $\sqrt{Q} = P$.
	The if-direction follows.
	
	For the only-if-direction we utilize the isomorphism in Remark~\ref{rem_isom_restrict_Weyl_mod} and Lemma~\ref{lem_prim_ideal_implies_bimod}.
	The condition (\ref{eq:leftequalsright})
	is equivalent to the bimodule condition in Lemma~\ref{lem_prim_ideal_implies_bimod}.
\end{proof}

The next theorem is a key  ingredient in the Ehrenpreis-Palamodov theory. 
Our result in the previous section provides a proof that is of independent interest.

\begin{theorem}[Noetherian operators] \label{thm_Noeth_ops}
	For every $P$-primary ideal $Q$ of multiplicity $m$ over $P$, there exist $A_1,\ldots,A_m$
	in the relative Weyl algebra $D_{n,c}$ such that (\ref{eq:fromAtoQ}) holds.
\end{theorem}

\begin{proof}
	Theorem \ref{thm_Noeth_ops} follows from Theorem~\ref{thm:main},
	to be proved in the next three sections.~Indeed, if we are given a $P$-primary ideal $Q$ of multiplicity $m$ over $P$, then  $Q$ %corresponds to 
	specifies~an $m$-dimensional $R$-bimodule inside the $\FF$-vector space
	$\FF \otimes_R D_{n,c}$. We choose elements
	$A_1,\ldots,A_m$ in $D_{n,c}$ whose images form an
	$\FF$-basis for that $R$-bimodule.
	These operators satisfy (\ref{eq:fromAtoQ}).
\end{proof}

Following Palamodov \cite{PALAMODOV}, we call $A_1,\ldots,A_m$ the {\em Noetherian operators} that encode the primary ideal $Q$.
It is an essential feature that these are linear differential operators with polynomial coefficients.
Operators with constant coefficients do not suffice. 
In other words, the Weyl algebra is essential in describing primary ideals. 
This key point is due to Palamodov. 
It had  been overlooked initially by Gr\"obner and Ehrenpreis.
For instance, consider the ideal $Q$ for $n=4, m=3$ in the Introduction. 
Three Noetherian operators $A_1,A_2,A_3$ are given in (\ref{eq:twistedcubic6}), and it is instructive to verify condition~(\ref{eq:leftequalsright}).
Algorithms for passing back and forth between Noetherian operators and ideal generators of $Q$ will be presented in Section \ref{sec8}.

Our problem is to solve a homogeneous system of linear PDE with constant coefficients. 
This is given by the generators of a primary ideal $Q$ in $\KK[x_1,\ldots,x_n]$,
where $x_j$ stands for the differential operator $\partial_{z_j} = \partial / \partial z_j$ with respect to a new unknown $z_j$.
Our aim is to characterize all sufficiently differentiable functions $\psi(z_1,\ldots,z_n)$ that are solutions to these PDE.
This characterization is the content of the Ehrenpreis-Palamodov Theorem, to be stated below.
Note that, if we are given an arbitrary system $J \subset R$ of such PDE then we can reduce to the case discussed here by computing a primary decomposition of the ideal $J$.

For the analytic discussion that follows, we work over the field $\KK = \CC$ of complex numbers.
Suppose $Q = \langle p_1,p_2,\ldots,p_r \rangle$, where $p_k = p_k(\mathbf{x})$.
This determines a  system of $r$ linear PDE:
\begin{equation}
\label{eq:mustsolvethis}
\qquad p_k(\partial_\mathbf{z})\bullet \psi(\mathbf{z})\,\, =\,\, 0 \qquad \text{ for }  k=1,2,\ldots,r.
\end{equation}
Let $\mathcal{K} \subset \RR^n$ be a compact convex set.
We seek all functions 
$\psi(\mathbf{z})$ in
$\, C^\infty(\mathcal{K})\,$ that satisfy (\ref{eq:mustsolvethis}).
Here we also use vector notation, namely $\mathbf{z} = (z_1,\ldots,z_n)$ and
$\partial_\mathbf{z} = (\partial_{z_1},\ldots,\partial_{z_n})$.
According to Theorem \ref{thm_Noeth_ops}, there exist Noetherian operators
$A_1(\mathbf{x},\partial_\mathbf{x}),\ldots,A_m(\mathbf{x},\partial_\mathbf{x})$
which encode the primary ideal $Q$ in the sense of (\ref{eq:fromAtoQ}).
In symbols, $\,	A_l( \mathbf{x}, \partial_\mathbf{x}) \bullet f \in P\,$ for all $l$.

Each $A_l$ is an element in the relative Weyl algebra $D_{n,c}$, given
as a unique $\CC$-linear combination of standard monomials
$\,\mathbf{x}^\alpha \partial_\mathbf{x}^\beta $.  This is important
since $D_{n,c}$ is non-commutative.  We now replace $\partial_\mathbf{x}$ by $\mathbf{z}$
in the standard monomials. This results in commutative polynomials 
\begin{equation}
\label{eq:thisresults}
B_l (\mathbf{x},\mathbf{z}) \,\, := \,\, 
A_l(\mathbf{x},\partial_\mathbf{x})|_{\partial_{x_1} \mapsto z_1,\ldots,
	\partial_{x_c} \mapsto z_c}
\qquad {\rm for} \quad l=1,2,\ldots,m .
\end{equation}
We call $B_1,\ldots,B_m$ the {\em Noetherian multipliers}
of the primary ideal $Q$. These
are polynomials in $n+c$ variables, obtained by reinterpreting the
Noetherian (differential) operators.
Note that $B_1,\ldots,B_m$ span
the  inverse system in Theorem \ref{thm:main}~(c) when viewed inside $\FF[z_1,\ldots,z_c]$.

\begin{example}
	The Noetherian operators and  multipliers in the Introduction~are
	\begin{equation}
	\label{eq:NoetMult}
	\begin{matrix}
	A_1 \,=\, 1\,,\;\, A_2\,=\, \partial_{x_1}
	\,\,\,{\rm and} \,\,\,A_3 \,=\, \partial_{x_1}^2 \,-\, 2 \,x_2\,\partial_{x_2}, \\
	B_1 \,=\, 1\,,\;\,\, B_2\,=\, z_1\,
	\,\,\,{\rm and}\, \,\,\,B_3 \,=\, z_1^2 \,-\, 2 \,x_2\,z_2. \qquad
	\end{matrix}
	\end{equation}
	This is consistent with (\ref{eq:twistedcubic5}) because
	$x_2 = st^2$ holds on the variety $V(P)$. 
\end{example}

Here is now the celebrated result on  linear PDE with constant coefficients:

\begin{theorem}[Ehrenpreis-Palamodov Fundamental Principle]
	\label{thm:Palamodov_Ehrenpreis}
	Fix the system (\ref{eq:mustsolvethis}) of PDE
	given by the $P$-primary ideal $Q$.
	Any solution $\psi$   in $C^\infty(\mathcal{K})$ has an integral representation
	\begin{equation}
	\label{eq:anysolution}
	\psi(\mathbf{z}) \,\,\,= \,\,\, \sum_{l=1}^m\,\int_{V(P)} \!\! B_l\left(\mathbf{x},\mathbf{z}\right) 
	\exp\left( \mathbf{x}^t \,\mathbf{z} \right) d\mu_l(\mathbf{x})
	\end{equation}
	for suitable measures $\mu_l$ supported in $V(P)$.
	Conversely, such functions are solutions.
\end{theorem}

We refer to \cite{BJORK,EHRENPREIS,PALAMODOV} for
the precise statement and its proof. In what follows we give a brief outline of the key idea.
	We follow the conventions used in analysis
	(cf.~\cite[Chapter 8]{BJORK}) and we write our system in terms of
	the differential operators $D_{z_j} = -i\partial_{z_j}$, where
	$i=\sqrt{-1}$. We can account for this in the Noetherian multipliers 
	by replacing $\mathbf{x}$ with $-i \mathbf{x}$. It is shown in
	\cite[Theorem 1.3, page 339]{BJORK} that any solution
	in $C^\infty(\mathcal{K})$ to the system (\ref{eq:mustsolvethis})  can be written~as
	$$ \psi(\mathbf{z}) \,\,=\,\, \sum_{l=1}^m\int_{V(P)} B_l\left(-i\mathbf{x},\mathbf{z}\right)
	\exp\left(-i \mathbf{x}^t \, \mathbf{z} \right) d\mu_l(\mathbf{x}). 		$$ 	
	We can now change variables, by incorporating the multiplication with $-i$ into the measures,	to get the formula  (\ref{eq:anysolution}).
	Conversely, to see that any such integral $\psi(\mathbf{z})$ is a solution to the PDE (\ref{eq:mustsolvethis}) given by $Q$, we differentiate under the integral sign and 				use the Fourier transform.

Theorem \ref{thm:Palamodov_Ehrenpreis} offers a finite 
representation of the infinite-dimensional space of solutions to any system of
linear PDE with constant coefficients. To reach this representation,
the given system is first decomposed into its
primary components. For each primary ideal $Q$, we then
 compute the Noetherian multipliers $B_1,\ldots,B_m$.
 The $B_i$ are polynomials in $({\bf x},{\bf z})$ that form 
 an $\mathbb{F}$-basis of the inverse system  in
 part (c) of Theorem~\ref{thm:main}. The computation can be carried
out in practise by running our {\tt Macaulay2} code that is
 described in Section \ref{sec8}. One uses Algorithm \ref{alg:forward}, 
 followed by a final step that modifies the
output $A_1,\ldots,A_m$ as in (\ref{eq:thisresults}).

\begin{example}
	\label{ex:NoetMult}
	Consider the  PDE determined by the ideal $Q$ in 
	the Introduction. The Noetherian multipliers in
	(\ref{eq:NoetMult}) furnish integral representations for all  solutions:
	\begin{align*}
			\psi(\mathbf{z}) \,\,= \,\,
		\int_{V(P)} \!\!\!\! \exp\left(\mathbf{x}^t \mathbf{z}\right) d\mu_1(\mathbf{x}) \,&+\,
		\int_{V(P)}\!\!\!\!  z_1 \exp\left(\mathbf{x}^t \mathbf{z}\right) d\mu_2(\mathbf{x}) \, \\&+\,
		\int_{V(P)} \!\!\!\! (z_1^2-2x_2 z_2) \exp\left(\mathbf{x}^t \mathbf{z}\right) d\mu_3(\mathbf{x}).
	\end{align*}
	Here $\mu_1,\mu_2,\mu_3$ are measures supported on the variety
	$\,V(P) = \bigl\{ \,(s^2t, s t^2,s^3, t^3) \,:\, s,t \in \CC \,\bigr\}$.
	The assertion in (\ref{eq:twistedcubic5}) is obtained by 
	pulling the integrals back to the 
	$(s,t)$-plane via the parametrization of $V(P)$.
	This replaces the measures $\mu_i$ by their pull-backs to that plane.
For a concrete solution take $\mu_1=\mu_2=0$ and
$\mu_3$ the	Dirac measure at $(2,3)$. In analogy to the step below (\ref{eq:twistedcubic4}),
this yields the solution
$\,\psi(\mathbf{z}) = (z_1^2 - 36 z_2) \,{\rm exp} (12 z_1 + 18 z_2 + 8 z_3 + 27 z_4)$.
\end{example}

We close by stating some differential equations that are more difficult to solve.
They depend on a parameter $k$, and the challenge arises when $k$ increases.
This example will be used in Section \ref{sec8} to illustrate
our algorithms and to demonstrate the scope of our implementation.

\begin{example} \label{ex:september17}
For any integer $k \geq 1$,
we are interested in functions $\psi(z_1,z_2,z_3,z_4)$ that~are annihiliated by
the $k$-fold application of the differential operators in (\ref{eq:twistedcubic2}). 
Thus, our system~is
$$
\biggl(\frac{\partial^2}{ \partial z_1^2} - \frac{\partial^2}{ \partial z_2 \partial z_3} \biggr)^{\! k} \! \psi({\bf z})
\,\, = \,\,
\biggl(\frac{\partial^2}{ \partial z_1 \partial z_2} - \frac{\partial^2}{ \partial z_3 \partial z_4} \biggr)^{\! k} \! \psi({\bf z})
\,\, = \,\,
\biggl(\frac{\partial^2}{ \partial z_2^2} - \frac{\partial^2}{ \partial z_1 \partial z_4} \biggr)^{\! k} \! \psi({\bf z})
\,\, = \,\, 0 .
$$
We want solutions $\psi({\bf z})$ that are {\em non-degenerate} in the sense
that $\psi({\bf z})$ cannot be annihilated by repeated differentiation. 
In symbols, we impose the restriction
\begin{equation}
\label{eq:restriction}
 \frac{ \partial^{i_1+\cdots+i_4}\, \psi}{\partial z_1^{i_1} \cdots \partial z_4^{i_4}} \,\,
\not= \,\, 0 \quad \hbox{for any} \,\,\, i_1,\ldots,i_4 \in \NN .
\end{equation}
The case $k=1$ is covered by (\ref{eq:twistedcubic4}).
In Example \ref{ex:september21} we explore the
 solutions for $k \geq 2$.

To model this problem with commutative algebra,
we start with the ideal
$$ 
J \,\, = \,\, \bigl\langle \,(x_1^2 - x_2 x_3)^k ,\,
(x_1 x_2 - x_3 x_4)^k, \,
(x_2^2 - x_1 x_4)^k \,\bigr\rangle.
$$
The radical of $J$ is the prime ideal $P$ of the twisted cubic.
The ideal $J$ also has 
three embedded primes, namely 
$\, \langle x_1, x_2, x_3 \rangle$, 
$\, \langle x_1, x_2, x_4 \rangle\,$ and
$\, \langle x_1, x_2, x_3,x_4 \rangle $.

We shall apply Theorem \ref{thm:Palamodov_Ehrenpreis}
to the $P$-primary component of $J$. This is the ideal 
\begin{equation}
\label{eq:ourintention}
Q \,\,=\,\, J : \langle x_1 x_2 x_3 x_4 \rangle^\infty.
\end{equation}
This saturation step models the restriction (\ref{eq:restriction}) to non-degenerate solutions.
\end{example}

\section{Hilbert Schemes and Inverse Systems}
\label{sec4}

Sections \ref{sec4}, \ref{sec5} and \ref{sec6} are devoted to the proof of our main result.
The details are quite technical. Complete understanding will require considerable experience in commutative algebra.
  In this section we provide a proof of the bijections between parts (a), (b) and (c) of Theorem~\ref{thm:main}. Here the key players are punctual 
Hilbert schemes and Macaulay's inverse systems.

We retain the notation from Sections~\ref{sec2} and \ref{sec3},
and we write $\pp=PS$ for the extension of our prime ideal $P$ in
$R = \KK[x_1,\ldots,x_c,x_{c+1},\ldots,x_n]$  to $S=\KK(x_{c+1},\ldots,x_n)[x_1,\ldots,x_c]$.
As before, we fix a maximal independent set of variables. After permuting variables, this set is
$\{x_{c+1},\ldots,x_n\}$. Thus, the field extension 
  $\KK(x_{c+1},\ldots,x_n) \hookrightarrow \FF = \text{Quot}(R/P)$ is algebraic.
 This implies that $\pp$ is a maximal ideal in $S$. 
Our first goal is to parametrize $P$-primary ideals of fixed multiplicity $m$ over $P$ by the punctual Hilbert scheme ${\rm Hilb}^m\bigl(\FF[[y_1,\ldots,y_c]]\bigr)$.
A special role is played by the inclusion map 
$\gamma:R \hookrightarrow \FF[y_1,\dots,y_c]$ in
(\ref{eq_map_gamma}). This induces an inclusion
$\,\gamma_S : S \hookrightarrow \FF[y_1,\ldots,y_c]$, also given by
$\,x_i \mapsto y_i+u_i $ for $i \leq c\,$
and $\,x_j \mapsto u_j$ for $j > c$.

\begin{remark}
	\label{rem_local_primary_ideals}
	Since $\KK[x_{c+1},\ldots,x_n] \cap P=0$, the canonical map $R \hookrightarrow S$ gives a bijection
	between $P$-primary ideals and $\pp$-primary ideals (see, e.g., \cite[Theorem 4.1]{MATSUMURA}).
\end{remark}

The homogeneous maximal ideal in $\FF[y_1,\dots,y_c]$ is 
denoted by $\MM=\langle y_1,\ldots,y_c\rangle$.
For any $f(\mathbf{x})=f(x_1,\ldots,x_n) \in P$, 
we have $f(\mathbf{u})=f(u_1,\ldots,u_n)=0$ in $ \FF$.
A Taylor expansion yields  
$$
f(\mathbf{u}+\mathbf{y})
\,\,=\,\,f(u_1+y_1,\ldots,u_c+y_c,u_{c+1},\ldots,u_n)
\,\,=\,\, \sum_{\substack{\lambda \in \NN^c\\\lvert\lambda \rvert > 0}} \frac{1}{\lambda_1!\cdots\lambda_c!}\frac{\partial^{\lvert\lambda\rvert} f}{\partial_{x_1}^{\lambda_1}\cdots\partial_{x_c}^{\lambda_c}}(\mathbf{u})\,\mathbf{y}^\lambda.
$$
This shows that $\gamma(P) \subseteq \MM$, and therefore
$\gamma_S( \pp) \subseteq \MM$.
The next proposition establishes a bijection 
between $\pp$-primary ideals containing $\pp^m$ and $\MM$-primary ideals containing $\MM^m$.

\begin{proposition}
	\label{prop_corespondence_primary_ideals}
	For all $m \ge 1$, the inclusion $\gamma_S$ induces 
	the isomorphism of local rings 
	$$
	S/\pp^m \,\,\xrightarrow{\cong}\,\, \FF[y_1,\ldots,y_c]/\MM^m.
	$$
\end{proposition}

\begin{proof}
	This result also appeared in \cite[Proposition 4.1]{BRUMFIEL_DIFF_PRIM} and \cite[Proposition 3.9]{NOETH_OPS}. In these sources the field
	 $\KK$ is assumed to be  perfect.  This holds here since ${\rm char}(\KK) {=} 0$.
\end{proof}

\begin{remark}
	\label{rem_ident_Hilb}
	(i) An ideal of colength $m$ in $\FF[[y_1,\ldots,y_c]]$ 
	contains the ideal $\langle y_1,\ldots,y_c \rangle^m $.
	So, ${\rm Hilb}^m\bigl(\FF[[y_1,\ldots,y_c]]\bigr)$ can be identified with 
	$\,{\rm Hilb}^m\bigl(\FF[[y_1,\ldots,y_c]]/ \langle y_1,\ldots,y_c \rangle^m \bigr).	$
	\noindent
	(ii) Any $ \langle y_1,\ldots,y_c\rangle$-primary ideal 
	of colength $m$ in the polynomial ring $\FF[y_1,\ldots,y_c]$ 
	contains the ideal $\langle y_1,\ldots,y_c \rangle^m $.
		For all $m>0$, we have the natural isomorphism 
	$$	\frac{\FF[[y_1,\ldots,y_c]]}{{ \langle y_1,\ldots,y_c \rangle}^m}
	\,\,\cong \,\,\frac{\FF[y_1,\ldots,y_c]}{{ \langle y_1,\ldots,y_c \rangle }^m}.	$$
	Hence, the $\langle y_1,\ldots,y_c \rangle$-primary ideals 
	of colength $m$ in $\FF[y_1,\ldots,y_c]$  are parametrized by 
	${\rm Hilb}^m \bigl( \FF[[y_1,\ldots,y_c]] \bigr)$.	
	From now on, $\langle y_1,\ldots,y_c \rangle $-primary ideals in the polynomial ring
	$\FF[y_1,\ldots,y_c]$ will automatically be identified with
	ideals in the power series ring $\FF[[y_1,\ldots,y_c]]$.
\end{remark}

We now prove the correspondence between 
parts (a) and (b) in Theorem~\ref{thm:main}.

\begin{theorem}
	\label{thm:param_primary}
	As asserted in (\ref{eq:corr12}), there is a bijective correspondence 
	$$
	\begin{array}{ccc}
	\left\lbrace\begin{array}{c}
	\mbox{$P$-primary ideals of $R$}\\
	\mbox{with multiplicity $m$ over $P$}
	\end{array}\right\rbrace
	
	& \longleftrightarrow &
	
	\left\lbrace\begin{array}{c}
		\mbox{points in }{\rm Hilb}^m(\FF[[y_1,\ldots,y_c]])\\
	\end{array}\right\rbrace\\
	
	Q & \longrightarrow & I=\langle y_1,\dots,y_c\rangle^m+\gamma(
	Q)\FF[y_1,\dots,y_c]\\
	Q=\gamma^{-1}(I) & \longleftarrow & I.
\end{array}
$$
\end{theorem}

\begin{proof}
	The canonical map $R \hookrightarrow S$ gives a bijection between $P$-primary ideals and $\pp$-primary ideals (Remark~\ref{rem_local_primary_ideals}).
	Also, for any $P$-primary ideal $Q \subset R$ we have $R_P/QR_P \cong S_{\pp}/Q S_{\pp}$.
	So, nothing is changed if we take $S$ and $\pp$ instead of $R$ and $P$.
	We have the commutative~diagram 
	\begin{center}
		\begin{tikzpicture}[baseline=(current  bounding  box.center)]
		\matrix (m) [matrix of math nodes,row sep=3em,column sep=9em,minimum width=1.7em, text height=1.5ex, text depth=0.25ex]
		{
			S &  \FF[\mathbf{y}] \\
			S/\pp^{m} &  \FF[\mathbf{y}]/\MM^{m}. \\
		};
		\path[-stealth]
		(m-1-1) edge node [above]	{$\gamma_S$} (m-1-2)
		(m-2-1) edge node [above]	{$\cong$} (m-2-2)
		;		
		\path [draw,->>] (m-1-1) -- (m-2-1);
		\path [draw,->>] (m-1-2) -- (m-2-2);
		\end{tikzpicture}	
	\end{center}
	The map in the bottom row is the isomorphism in
	Proposition~\ref{prop_corespondence_primary_ideals}.
	This gives an inclusion-preserving bijection between $\pp$-primary ideals containing $\pp^m$ and $\MM$-primary ideals containing $\MM^{m}$, in particular, colength does not change under this correspondence.
	In explicit terms, the $\MM$-primary ideal $I$ corresponding to a $\pp$-primary ideal $QS \supseteq \pp^m$ is 
	$$	I \,\,=\,\, \MM^{m} \,+\,  \gamma_S(QS)\big(\FF[\mathbf{y}]\big).	$$
	And, the $\pp$-primary ideal $QS$ corresponding to an $\MM$-primary ideal $I \supseteq \MM^m$ is 
	$$	QS \,\,=\,\, \gamma_S^{-1}(I).	$$
	Finally, the result now follows from Remark~\ref{rem_ident_Hilb}.
\end{proof}

We next show the correspondence between parts (b) and (c) in Theorem~\ref{thm:main}.
This  follows from the usual Macaulay duality.
Although this argument is well-known, we will need a short discussion to later 
connect parts (c) and (d) of Theorem~\ref{thm:main}.
Consider the injective hull $E=E_{\FF[[y_1,\ldots,y_c]]}(\FF)$ of the residue field $\FF \cong \FF[[y_1,\ldots,y_c]]/\langle y_1,\ldots,y_c \rangle$ of $\FF[[y_1,\ldots,y_c]]$.
Since $\FF[[y_1,\ldots,y_c]]$ is a formal power series ring, this equals the  
module of inverse polynomials:
\begin{equation}
\label{eq_isom_E_inv_sys}
E \,\,\cong \,\, \FF[y_1^{-1},\ldots,y_c^{-1}].
\end{equation}
For a derivation see 
e.g.~\cite[Lemma 11.2.3, Example 13.5.3]{Brodmann_Sharp_local_cohom}
or \cite[Theorem 3.5.8]{BRUNS_HERZOG}.

Consider  $\FF[z_1,\ldots,z_c]$  as an $\FF[[y_1,\ldots,y_c]]$-module by setting that $y_i$ acts on $\FF[z_1,\ldots,z_c]$ as $\partial_{z_i}$, that is, $y_i \cdot F = \partial_{z_i} \bullet F$ for any $F \in \FF[z_1,\ldots,z_c]$.
Since  $\FF$ has characteristic zero, we have the following isomorphism of $\FF[[y_1,\ldots,y_c]]$-modules:
\begin{equation}
\label{eq_isom_inv_sys_char_zero}
\FF[y_1^{-1},\ldots,y_c^{-1}] \;\xrightarrow{\cong}\; \FF[z_1,\ldots,z_c], \quad \frac{1}{\mathbf{y}^{\alpha}} \;\mapsto\;  \frac{\mathbf{z}^\alpha}{\alpha!}.
\end{equation}

Now, Macaulay's duality is simply performed via Matlis duality.
We use ${\left(-\right)}^\vee$ to denote Matlis dual ${\left(-\right)}^\vee=\Hom_{\FF[[y_1,\ldots,y_c]]}\left(-,E\right)$.
This is  a contravariant exact functor which establishes an anti-equivalence between the full-subcategories of artinian $\FF[[y_1,\ldots,y_c]]$-modules and finitely generated $\FF[[y_1,\ldots,y_c]]$-modules (see, e.g., \cite[Theorem 3.2.13]{BRUNS_HERZOG}).
For any zero-dimensional ideal $I $ in $
\FF[[y_1,\ldots,y_c]]$, the isomorphisms (\ref{eq_isom_E_inv_sys}) and (\ref{eq_isom_inv_sys_char_zero}) together with Matlis duality
yield the  identifications
$$
I^{\perp} \,=\, \left\lbrace F\in \FF[z_1,\ldots,z_c]: f\bullet F=0 \mbox{ for all }
f\in I\right\rbrace \,\cong \, (0 :_{E} I) \,\cong \, \bigl(\FF[[y_1,\ldots,y_c]]/I \bigr)^\vee.
$$
On the other hand, consider any $\FF[[y_1,\ldots,y_c]]$-submodule $V$
of $ \FF[z_1,\ldots,z_c] \cong E$. Then $V$ is an $\FF$-subspace of $\FF[z_1,\ldots,z_c]$ that is closed by differentiation, 
as $y_i$ is identified with the operator $\partial_{z_i}$.
Again, the isomorphisms (\ref{eq_isom_E_inv_sys}) and (\ref{eq_isom_inv_sys_char_zero}) 
with Matlis duality give  identifications
$$
{\rm Ann}_{\FF[[\partial_{z_1},\dots,\partial_{z_c}]]}(V) \,\, \cong \,\, 
{\rm Ann}_{\FF[[y_1,\ldots,y_c]]}(V) \, \, \cong \,\,
{\bigl( E /V \bigr)}^{\vee} \, \,\subset \,\, \FF[[y_1,\ldots,y_c]].
$$
Hence, our discussion yields the connection between
(b) and (c) in Theorem~\ref{thm:main}.

\begin{theorem}[Macaulay's duality]
	\label{thm:Macaulay_dual}
	As asserted in (\ref{eq:HilbBijection}),
	there is a bijection
		$$
	\begin{array}{ccc}
	\left\lbrace\begin{array}{c}
	\mbox{points in }{\rm Hilb}^m\left(\FF[[\partial_{z_1},\ldots,\partial_{z_c}]]\right)\\
	\end{array}\right\rbrace
	
	& \longleftrightarrow &
	
	\left\lbrace\begin{array}{c}
		\mbox{$m$-dimensional $\FF$-subspaces of}\\
		\mbox{$\FF[z_1,\dots,z_c]$ closed by differentiation}
	\end{array}\right\rbrace\\

	I & \longrightarrow & V=I^\perp\\
	I={\rm Ann}_{\FF[[\partial_{z_1},\dots,\partial_{z_c}]]}(V) & \longleftarrow & V.\\
\end{array}
$$
\end{theorem}

\section{Differential Operators Revisited}
\label{sec5}

In this section we review some material on
differential operators in commutative algebra.
This is used in Section \ref{sec6} to complete
the proof of Theorem~\ref{thm:main}. Even though the
Noetherian operators  $A_i$ live in the Weyl algebra, 
we need the abstract perspective to link them to
the Weyl-Noether module (\ref{eq:relativeweyl}).
As before, $\KK$ is a field of characteristic zero and $R=\KK[x_1,\ldots,x_n]$.

Given $R$-modules $M$ and $N$, we regard $\Hom_\KK(M, N)$ as an $(R\otimes_\KK R)$-module~via
$$
\left((r \otimes_\KK s) \delta\right)(w)\,\, =\,\, r \delta(sw) \quad \text{ for all }\,\, \delta \in \Hom_\KK(M, N), \; w \in M,\; r,s \in R. 
$$ 
This is equivalent to saying that $\Hom_\KK(M,N)$ is an $R$-bimodule, where the action on the left is given by post-composing $(r \cdot \delta)(w)=r\delta(w)$ and the action on the right is given by pre-composing $(\delta \cdot s)(w)=\delta(sw)$, for all $\delta \in \Hom_\KK(M, N), \; w \in M,\; r,s \in R$.
We use the bracket notation $[\delta,r](w) = \delta(rw)-r\delta(w)$ for all $\delta \in \Hom_\KK(M, N)$, $r \in R$ and $w \in M$.

\begin{notation}
	\label{nota_T_mod_struct}
	We write $T=R \otimes_\KK R = \KK[x_1, \ldots, x_n, y_1, \ldots, y_n]$ as a polynomial ring in $2n$ variables, where $x_i$ represents $x_i \otimes_\KK 1$ and $y_i$ represents $1 \otimes_\KK x_i - x_i \otimes_\KK 1$.
	The action of $T$ on $\Hom_\KK(M,N)$ is thus as follows.
	For all $\delta \in \Hom_\KK(M, N)$ and $w \in M$, 
	$$
	(x_i \cdot \delta) (w) = x_i \delta(w) \; \text{ and } \;  (y_i \cdot \delta) (w)  = \delta(x_i w) - x_i\delta(w) = \left[\delta,x_i\right](w) \;\; \text{ for} \,\,i=1,\ldots,n.
	$$
\end{notation}

Any $T$-module is an $R$-module via the canonical map $R \hookrightarrow T, x_i  \mapsto x_i $.
Thus, any $T$-module has an $R$-module structure by using the left factor $R\otimes_\KK 1 \subset T = R \otimes_\KK R$.
The $\KK$-linear differential operators form a $T$-submodule of $\Hom_\KK(M, N)$,
as follows.

\begin{definition}
	\label{def_diff_ops}
	Let $M, N$ be $R$-modules.
	The \textit{$m$-th order $\KK$-linear differential operators} $\Diff_{R/\KK}^m(M, N) \subseteq \Hom_\KK(M, N)$ 
	form a $T$-module that is	 defined inductively~by:
	\begin{enumerate}[(i)]
		\item $\Diff_{R/\KK}^{0}(M,N) \,:=\, \Hom_R(M,N)$.
		\item $\Diff_{R/\KK}^{m}(M, N) \,:= \,
		\big\lbrace \delta \in \Hom_\KK(M,N) : [\delta, r] \in \Diff_{R/\KK}^{m-1}(M, N) 
		\,\text{ for all }\, r \in R \big\rbrace$.
	\end{enumerate}
	The set of all \textit{$\KK$-linear differential operators from $M$ to $N$} 
	is the $T$-module
	$$
	\Diff_{R/\KK}(M, N) \,\,:=\,\, \bigcup_{m=0}^\infty \Diff_{R/\KK}^m(M,N).
	$$
	Subsets $\mathcal{E} \subseteq \Diff_{R/\KK}(M, N)$ are viewed
	as differential equations, with solution~spaces
	\begin{equation}
	\label{eq:solE}
	{\rm Sol}(\mathcal{E}) \,\,:= \,\,\big\lbrace w \in M : \delta(w) = 0 \text{ for all } \delta \in \mathcal{E} \big\rbrace 
	\,\,= \,\,\bigcap_{\delta \in \mathcal{E} } {\rm Ker}(\delta).
	\end{equation}
\end{definition}

Following the approach in \cite[Section 2]{NOETH_OPS},
we now introduce  the module of principal parts.
By construction,
the ideal $ \Delta_{R/\KK} = \langle y_1, \ldots,y_n \rangle$
in $T$ is the kernel of the multiplication map 
$$
T = R \otimes_\KK R \,\rightarrow \,R\,, \quad 	r \otimes_\KK s \,\mapsto\, rs.
$$

\begin{definition}
	Given an $R$-module $M$,
	the module of \textit{$m$-th principal parts of $M$} equals
	$$
	P_{R/\KK}^m(M) \,\,:=\,\, \frac{R \otimes_\KK M}{\Delta_{R/\KK}^{m+1}  \left(R \otimes_\KK M\right)}.
	$$
	This is a  $T$-module. It has the natural map
	$\,d^m : \,M \rightarrow P_{R/\KK}^m(M),\,
	w \mapsto \overline{1 \otimes_\KK w}$.
	For $M=R$ we abbreviate
	$\,P_{R/\KK}^m \,:=\, P_{R/\KK}^m(R)=T/\Delta_{R/\KK}^{m+1}$,
	and the map becomes
	\begin{equation}
	\label{eq_univ_diff}
	d^m : R \rightarrow P_{R/\KK}^m, \;\;x_i \, \mapsto \,\overline{1 \otimes_\KK x_i} \,=\, \overline{x_i+y_i} .
	\end{equation}	
\end{definition}

The following result is a fundamental characterization of differential operators.

\begin{proposition}[{\cite[Proposition 16.8.4]{EGAIV_IV}, 
		\cite[Theorem 2.2.6]{AFFINE_HOPF_I}}]
	\label{prop_represen_diff_opp}
	Let $m\ge 0$ and let $M, N$ be $R$-modules.
	Then, the following map is an isomorphism of $R$-modules:
	\begin{align*}
	{\left(d^m\right)}^*\, :\, \Hom_R\left(P_{R/\KK}^m(M), N\right) 
	&\,\,\xrightarrow{\cong} \,\,\Diff_{R/\KK}^m(M, N), \\
	\varphi  \quad &\,\,\mapsto \,\quad \varphi \circ d^m.
	\end{align*}
\end{proposition}

This is a very general result for commutative rings $R$. What we are interested in here
is the polynomial ring $R = \KK[x_1,\ldots,x_n]$ over a field $\KK$
of characteristic zero. In this case, the $R$-module
$P_{R/\KK}^m=T/\Delta_{R/\KK}^{m+1}$ is free, and a basis
is given by $\bf y$-monomials of degree at most $m$:
\begin{equation}
\label{eq_direct_sum_Prin}
P_{R/\KK}^m \,\,\,= \,\,\bigoplus_{\lvert \alpha \rvert \le m} R\mathbf{y}^\alpha \quad =
\bigoplus_{\alpha_1 + \cdots + \alpha_r \le m} \!\!\!\! R y_1^{\alpha_1}\cdots y_n^{\alpha_n}.
\end{equation}
Proposition~\ref{prop_represen_diff_opp} implies that $\,\Diff_{R/\KK}^m(R, R)\, \cong\,
\Hom_R\bigl(P_{R/\KK}^m, R\bigr)\,$ is a free $R$-module with  basis
\begin{equation}
\label{eq_basis_diff_ops}
\big\{ {(y_1^{\alpha_1}\cdots y_n^{\alpha_n})}^* \circ d^m  : \alpha_1 + \cdots + \alpha_n \le m \big\}.
\end{equation}
For any polynomial $f \in R$, the operator $d^m$ in
(\ref{eq_univ_diff}) computes the Taylor expansion
$$
d^m(f(\mathbf{x})) \,\,= \,\,f(1\otimes_\KK \mathbf{x}) \,\,= \,\,f(\mathbf{x}+\mathbf{y}) \,\,=\,\,
\sum_{\lambda \in \NN^n} \left(D_\mathbf{x}^\lambda f\right)\!(\mathbf{x})\,\mathbf{y}^\lambda,
$$
where $\,D_\mathbf{x}^{\lambda}:R\rightarrow R\,$ is
the differential operator we all know from calculus:
$$
D_\mathbf{x}^{\lambda}\,\, =\,\, \frac{1}{\lambda!}\partial_\mathbf{x}^\lambda \,\,= \,\,
\frac{1}{\lambda_1!\cdots \lambda_n!}\partial_{x^1}^{\lambda_1}\cdots \partial_{x^n}^{\lambda_n}.
$$
For any $\alpha \in \NN^n$ we  have $
\left({(\mathbf{y}^\alpha)}^* \circ d^m\right)(f(\mathbf{x})) = \left(D_\mathbf{x}^\alpha f\right)(\mathbf{x})$. Equation (\ref{eq_basis_diff_ops})~implies
$$ \Diff_{R/\KK}^m(R,R)\,\, =\,\,\bigoplus_{\lvert \alpha \rvert \le m} R D_\mathbf{x}^\alpha
\,\,=\,\, \bigoplus_{\lvert \alpha \rvert \le m} R \partial_\mathbf{x}^\alpha. $$
By letting $m$ go to infinity, we now recover the Weyl algebra in its well-known role:

\begin{lemma}	\label{lem_Weyl_as_diff_ops}
	$\,\Diff_{R/\KK}(R, R)$ coincides with the  Weyl algebra $\KK\langle x_1,\ldots,x_n,\partial_{x_1},\ldots,\partial_{x_n} \rangle$.
\end{lemma}

Let $J $ be an ideal in $R= \KK[x_1,\ldots,x_n]$. The canonical projection
$\pi : R \rightarrow R/J$ induces a natural map of differential operators.
This is the following homomorphism of $T$-modules:
\begin{equation}
\label{eq:pimap}
\Diff_{R/\KK}^m(\pi) \,: \,\Diff_{R/\KK}^m(R, R) \,\rightarrow \,
\Diff_{R/\KK}^m(R, R/J), \quad \delta \,\mapsto\, \pi \circ \delta.
\end{equation}

\begin{lemma}
	\label{lem_diff_ops_R/J}
	We have the following explicit description of the objects in (\ref{eq:pimap}):
	\begin{enumerate}[(i)]
		\item $\Diff_{R/\KK}^m(R, R/J)$ is a free $R/J$-module with direct summands decomposition 
		$$ \qquad
		\Diff_{R/\KK}^m(R,R/J) \,\,= \,\,\bigoplus_{\lvert \alpha \rvert \le m} (R/J) \overline{D_\mathbf{x}^\alpha}, \qquad {\rm where}\,\,\,
		\overline{D_\mathbf{x}^\alpha} = \pi \circ D_\mathbf{x}^\alpha.		$$
		\item The map $\Diff_{R/\KK}^m(\pi)$ is surjective. Explicitly, any differential operator 
		$$
		\epsilon \,\,=\, \sum_{\lvert \alpha \rvert \le m} \overline{r_\alpha}
		\overline{D_\mathbf{x}^\alpha} \,\in\, \Diff_{R/\KK}^m(R,R/J), \quad \text{ where } r_\alpha \in R,  
		$$ 
		lifts to an operator 
		$\,\delta=\sum_{\lvert \alpha \rvert \le m} r_\alpha D_\mathbf{x}^\alpha \,\in\, \Diff_{R/\KK}^m(R,R)\,$
		with $\,\epsilon = \Diff_{R/\KK}^m(\pi)(\delta)$.
	\end{enumerate}
\end{lemma}

\begin{proof}
	$(i)$ From Proposition~\ref{prop_represen_diff_opp} and the Hom-tensor adjunction we get  isomorphisms 
	\begin{align}
	\label{eq_isoms_Diff_R/J}
	\begin{split}
	\Hom_{R/J}\left(R/J \otimes_R P_{R/\KK}^m, R/J\right) &
	\,\,\cong \,\, \Hom_R\left(P_{R/\KK}^m, \Hom_{R/J}\left(R/J, R/J\right)\right)\\
	& \,\,\cong \,\,\Hom_R\left(P_{R/\KK}^m, R/J\right)\\
	&\,\,\cong \,\,\Diff_{R/\KK}^m(R,R/J).
	\end{split}			
	\end{align}
	The isomorphism from the first row to the second row in (\ref{eq_isoms_Diff_R/J})  is given by 
	$$	\psi \in \Hom_{R/J}\left(R/J \otimes_R P_{R/\KK}^m, R/J\right) \;\;\mapsto \;\;\psi\circ h_m \in \Hom_R\left(P_{R/\KK}^m, R/J\right),
	$$
	where $\,h_m\,$ is the canonical map $\, P_{R/\KK}^m \rightarrow R/J \otimes_R P_{R/\KK}^m$.
	Therefore, the isomorphism from the first to the third row in
	(\ref{eq_isoms_Diff_R/J}) is given explicitly as
	$ \,\psi\, \mapsto \, \psi \circ h_m \circ d^m $.
	From (\ref{eq_direct_sum_Prin}) we get that $R/J \otimes_R P_{R/\KK}^m$ is a free $R/J$-module with decomposition 
	$$ R/J \otimes_R P_{R/\KK}^m \,\,=
	\,\, \bigoplus_{\lvert \alpha \rvert \le m} (R/J)\mathbf{y}^\alpha. $$
	Our  isomorphisms
	 (\ref{eq_isoms_Diff_R/J})
	show that $\Diff_{R/\KK}^m(R,R/J)$ is a free $R/J$-module with~basis
	\begin{equation*}
	\big\{ {(y_1^{\alpha_1}\cdots y_n^{\alpha_n})}^* \circ h_m \circ d^m  : \alpha_1 + \cdots + \alpha_n \le m \big\}.
	\end{equation*}
	Now, for any  polynomial $f(\mathbf{x})$ in $ R$, we obtain the equations
	\begin{align} \label{eq_diff_opp_z_alpha}
				\begin{split}
	& \left({(\mathbf{y}^\alpha)}^* \circ h_m \circ d^m\right)(f(\mathbf{x})) \,\,
	= \,\,\left({(\mathbf{y}^\alpha)}^* \circ h_m\right)\left(\sum_{\lambda \in \NN^n} \left(D_\mathbf{x}^\lambda f\right)(\mathbf{x})\mathbf{y}^\lambda\right)  \\
	& \,\,=\,\, \left({(\mathbf{y}^\alpha)}^*\right)\left(\sum_{\lambda \in \NN^n} \pi\left(\left(D_\mathbf{x}^\lambda f\right)(\mathbf{x})\right)\mathbf{y}^\lambda\right)
	\,\,=\,\,\,\pi\big(\left(D_\mathbf{x}^\alpha f\right)(\mathbf{x})\big).
	\end{split}		
	\end{align}
	This implies that the operators $\overline{D_\mathbf{x}^\alpha}=\pi \circ D_\mathbf{x}^\alpha$ with $\lvert\alpha\rvert \le m$ give a basis of $\Diff_{R/\KK}^m(R, R/J)$.
	Part	 $(ii)$ follows from part $(i)$, concluding the proof of Lemma~\ref{lem_diff_ops_R/J}.
\end{proof}

\begin{remark}
	Since  $R$ is a polynomial ring, the process of lifting differential operators is easy and explicit.
	However, the surjectivity of $\Diff_{R/\KK}^m(\pi)$ is a subtle property, and it is not always satisfied over more general types of rings. More precisely, there exist primary ideals in non-polynomial rings that cannot be described by Noetherian operators as in (\ref{eq:fromAtoQ}), which implies that $\Diff_{R/\KK}^m(\pi)$ is not surjective in general. For an illustration see \cite[Example 5.2]{NOETH_OPS} and \cite[Proof of Corollary 3.13]{NOETH_OPS}.
\end{remark}

\section{Proof of the Representation Theorem} 
\label{sec6}

We  here complete the proof of Theorem~\ref{thm:main}. This is done
 by connecting part (d) on the Weyl-Noether module with the earlier parts (a), (b), and (c).
This section is divided into two subsections.
In the first subsection we treat the zero-dimensional situation, where $c=n$.
In the second subsection, we use a maximal independents subset modulo $P$ and the results on differential operators in Section~\ref{sec5} to reduce the general case to the zero-dimensional case.

\subsection{The zero-dimensional case}
We here restrict ourselves to ideals in $R= \KK[x_1,\ldots,x_n]$
that are primary to a maximal ideal $P$.
Hence $c=n$ and $\FF=R/P$. Since the base field $\KK$ is assumed to have
characteristic zero, an adaptation of Gr\"obner's classical approach via Macaulay's inverse system will be valid.

Writing $\,T=R\otimes_\KK R=\KK[x_1,\ldots,x_n,y_1,\ldots,y_n]\,$
as in Section~\ref{sec5}, we now have
\begin{equation}
\label{eq:FFRT}
\FF\otimes_R T \,\,=\,\, \FF\otimes_R \left(R\otimes_\KK R\right) \,\,\cong\,\,
R/P \otimes_\KK \KK[y_1,\ldots,y_n] \,\,\cong\,\, \FF[y_1, \ldots, y_n].
\end{equation}
This endows $\Diff_{R/\KK}^{m}\left(R, \FF\right)$ with the structure of an $\FF[y_1,\ldots,y_n]$-module.
Applying Lemma~\ref{lem_diff_ops_R/J} with $J=P$, we see that 
$\Diff_{R/\KK}^{m}\left(R, \FF\right)$ is a finite-dimensional $\FF$-vector space.
In the sequel, the homogeneous maximal ideal
$
\MM = \langle y_1,\ldots,y_n \rangle \subset \FF[y_1,\ldots,y_n]
$
will play an important role.
This ideal is also  $\MM=\Delta_{R/\KK} \left( \FF \otimes_R T \right)$.
For any $m \ge 0$ we identify
$$ \frac{\FF[y_1,\ldots,y_n]}{\MM^{m+1}} \,\,\,=\,\,
\bigoplus_{\vert\alpha\rvert \le m} \FF\mathbf{y}^{\alpha}.
$$
For any $\FF[y_1,\ldots,y_n]$-module $M$, the $\FF$-dual $\Hom_{\FF}(M, \FF)$ 
is naturally a module over $\FF[y_1,\ldots,y_n]$ as follows: if $\psi \in \Hom_{\FF}(M, \FF)$, then 
$y_i \cdot \psi$ is the $\FF$-linear map $\,\psi(y_i \cdot -)  :  w \in M \mapsto \psi(y_iw) \in \FF$.
The next result relates submodules of $\Diff_{R/\KK}^{m}\left(R, \FF\right)$ to
$\MM$-primary ideals in $\FF[y_1,\ldots,y_n]$.

\begin{proposition}
	\label{lem_descrip_diff_opp}
	The following statements hold for all positive integers $m$:
	\begin{enumerate}[(i)]
		\item We have an isomorphism of $\FF[y_1,\ldots,y_n]$-modules
		$$\Diff_{R/\KK}^{m-1}\left(R, \FF\right)\,\, \cong \,\,
		\Hom_{\FF}\Bigg(\frac{\FF[y_1,\ldots,y_n]}{{\MM}^{m}}, \FF\Bigg).
		$$
		
		\item The following map gives  a bijective correspondence 
		between $\MM$-primary ideals $I $ in $ \FF[y_1,\ldots,y_n]$ 
		that contain $\MM^{m}$ 
		and $\FF[y_1,\ldots,y_n]$-submodules of $\,\Diff_{R/\KK}^{m-1}\left(R,\FF\right)$:
		\begin{equation}
		\label{eq:ImapHom}
		I  \;\mapsto\; \Hom_{\FF}\left(\frac{\FF[y_1,\ldots,y_n]}{I}, \FF\right).
		\end{equation}
		\item 		
		Let $\mathcal{E} \subseteq \Diff_{R/\KK}^{m-1}(R,\FF)$
		be the image under (\ref{eq:ImapHom}) of an $\MM$-primary ideal $I \supseteq \MM^{m}$. Then,
		with notation as in (\ref{eq:solE}),
		$$ {\rm Sol}(\mathcal{E}) \,\,=\,\, \gamma^{-1}(I), $$
		where
		$\gamma$ is the inclusion $ R  \hookrightarrow \FF[y_1,\ldots,y_n], x_i \mapsto y_i{+}u_i$
		in~(\ref{eq_map_gamma}). \end{enumerate}
\end{proposition}

\begin{proof}
	This is essentially \cite[Lemma 3.8]{NOETH_OPS}.
	We provide a proof for completeness.
	
	$(i)$ Since $\FF=R/P$, from equation (\ref{eq_isoms_Diff_R/J}) we obtain the isomorphism 
	$$
	\Diff_{R/\KK}^{m-1}(R, \FF)\,\, \cong \,\,\Hom_{\FF}\bigl(\FF \otimes_R P_{R/\KK}^{m-1}, \FF\bigr).
	$$
	Thus, the result follows from the fact that 
	$\,\FF \otimes_R P_{R/\KK}^{m-1} \,\cong\, \FF[\mathbf{y}]/ \MM^{m}	$.
	
	$(ii)$ Since $\FF[\mathbf{y}]/\MM^{m}$ is finite-dimensional over $\FF$, the functor $\Hom_{\FF}\left(-,\FF\right)$ gives a bijection between  quotients of $\FF[\mathbf{y}]/\MM^{m}$ and $\FF[\mathbf{y}]$-submodules of ${\Hom_{\FF}\left(\frac{\FF[\mathbf{y}]}{\MM^{m}},\FF\right)}$. So, the claim follows from~$(i)$.
	
	$(iii)$ By assumption, $\,\mathcal{E} \,=\,{\Hom_{\FF}\left(\frac{\FF[\mathbf{y}]}{I},\FF\right)} \,$
	is in $\,\Diff_{R/\KK}^{m-1}(R,\FF)$. Consider the canonical map 
	$ \,\Phi_I : \frac{\FF[\mathbf{y}]}{\MM^{m}}  \twoheadrightarrow \frac{\FF[\mathbf{y}]}{I}\,$
	given by the $\mathcal{M}$-primary ideal $I \supseteq \MM^{m}$.
	The isomorphism (\ref{eq_isoms_Diff_R/J}) yields
	$$	{\rm Sol}(\mathcal{E})  \,\,=\,\,
	\bigl\{ f \in R : \left(\psi \circ \Phi_I \circ h_{m-1} \circ d^{m-1}\right)(f)=0 
	\text{ for all } \psi \in {\Hom_{\FF}\left(\FF[\mathbf{y}]/I ,\FF\right)} \bigr\}.
	$$
	The composition $\Phi_I \circ h_{m-1} \circ d^{m-1}$ equals the map  
	$\,R \mapsto \FF[\mathbf{y}]/ I,\, x_i \mapsto \overline{y_i+u_i}$. Hence
	\begin{align*}
	{\rm Sol}(\mathcal{E}) &\,\,=\,\, \bigl\{\,f \in R : 
	\psi\bigl(\,\overline{f(
		\mathbf{y}+\mathbf{u}})\, \bigr)=0\, \text{ for all } \,\psi \in {\Hom_{\FF}\left(\FF[\mathbf{y}]/ I ,\FF\right)}\bigr\} \\
	&\,\,=\,\, \bigl\{\, f \in R : f\bigl(\mathbf{y}+\mathbf{u}\bigr) \in I \, \bigr\}
	\,\,=\,\, \gamma^{-1}(I).
	\end{align*}
	This completes the proof of  Proposition \ref{lem_descrip_diff_opp}.
\end{proof} 

Next, under the assumption of $P$ being maximal, we relate part (d) with the other parts in Theorem~\ref{thm:main}.
By  Definition~\ref{def_diff_ops} and Lemma~\ref{lem_Weyl_as_diff_ops},  the 
Weyl-Noether module has the filtration 
$$
\FF \,\otimes_R \, R\langle\partial_{x_1},\ldots,\partial_{x_n}\rangle \,\;= \;\, \FF \, \otimes_R \, \biggl(\lim\limits_{\substack{\longrightarrow\\m}} \Diff_{R/\KK}^m(R,R) \biggr) \,\;\cong \; \,\lim\limits_{\substack{\longrightarrow\\m}}\Big(\FF \otimes_R \Diff_{R/\KK}^m(R,R)\Big).
$$
Applying  Lemma~\ref{lem_diff_ops_R/J} with $J=P$ gives  $\FF \otimes_R \Diff_{R/\KK}^m(R,R) \cong \Diff_{R/\KK}^m(R,\FF) \cong \bigoplus_{\vert\alpha\rvert\le m}\FF \overline{\partial_{\mathbf{x}}^\alpha}$. 
This gives rise to the following isomorphisms of $\FF$-vector spaces:
\begin{equation}
\label{eq_isom_relWeyl_diff}
\FF \,\otimes_R \, R\langle\partial_{x_1},\ldots,\partial_{x_n}\rangle \;\cong \; \Diff_{R/\KK}(R,\FF) \;\cong \; \bigoplus_{ \alpha  \in \NN^n} \FF \overline{\partial_\mathbf{x}^\alpha}.
\end{equation}
When the Weyl-Noether module was introduced 
in (\ref{eq:relativeweyl}), we gave a purely algebro-symbolic treatment 
and we noticed that an $\FF$-basis is given by  $\,\left\lbrace 1 \otimes_R \partial_{\bf x}^\alpha: \alpha \in\NN^n\right\rbrace$.
Now, with  the isomorphism (\ref{eq_isom_relWeyl_diff}), the elements $1 \otimes_R \partial_{\mathbf{x}}^\alpha $
are seen as the differential operators $\overline{\partial_{\mathbf{x}}^\alpha} \in \Diff_{R/\KK}(R, \FF)$.

The following map is an isomorphism of $\FF$-vector spaces:
\begin{equation}
\label{eq_map_omega}
\omega : \FF[z_1,\ldots,z_n] \;\rightarrow\; \FF \,\otimes_R \,  R\langle\partial_{x_1},\ldots,\partial_{x_n}\rangle\;\cong\;\Diff_{R/\KK}(R,\FF), \quad \mathbf{z}^\alpha \mapsto \partial_{\mathbf{x}}^\alpha.
\end{equation}
From (\ref{eq:FFRT}) and Notation~\ref{nota_T_mod_struct} we get the following actions.
For $\alpha \in \NN^n$ and $1 \le i \le n$,
\begin{equation}
\label{eq_deriv_z_bracket_partial}
\!\! \partial_{z_i} \bullet \mathbf{z}^\alpha = \alpha_iz_1^{\alpha_1}\cdots z_i^{\alpha_i-1}\!\!\cdots z_n^{\alpha_n} \,\text{ and }\, y_i\cdot\partial_{\mathbf{x}}^\alpha
=\left[\partial_{\mathbf{x}}^\alpha,x_i\right]=
\alpha_i\partial_{x_1}^{\alpha_1}\cdots \partial_{x_i}^{\alpha_i-1}\!\!\cdots \partial_{x_n}^{\alpha_n}.
\end{equation}
Hence the map $\omega$ in (\ref{eq_map_omega}) gives a bijection between $\FF$-vector subspaces of $\FF[z_1,\ldots,z_n]$ closed under differentiation and $\FF[y_1,\ldots,y_n]$-submodules of $\FF \otimes_R  R\langle\partial_{x_1},\ldots,\partial_{x_n}\rangle$. The latter
structure as a submodule is equivalent to being an $R$-subbimodule of the Weyl-Noether module.

\begin{lemma} \label{lem_prim_ideal_implies_bimod}
	Let $\mathcal{E}$ be a finite dimensional $\FF$-vector subspace of $ \Diff_{R/\KK}(R,\FF)$.
	If $\,Q = {\rm Sol}(\mathcal{E})\,$ is a $P$-primary ideal in $R = \KK[x_1,\ldots,x_n]$
	then $\mathcal{E}$ is an $R$-bimodule.
\end{lemma}
\begin{proof} Fix $m \in \NN$ such that $Q \supseteq P^m$ 
	and $\mathcal{E} \subseteq \Diff_{R/\KK}^{m-1}(R,\FF)$.
	The map $\gamma$ in (\ref{eq_map_gamma}) defines~the ideal 
	$I=\MM^m+\gamma(
	Q)\FF[y_1,\dots,y_n]$.
	Let $\mathcal{E}^\prime \subseteq \Hom_\FF\left(\frac{\FF[\mathbf{y}]}{\MM^m},\FF\right)$ 
	be the $\FF$-vector subspace~coming from $\mathcal{E}$
	under the isomorphism of Proposition~\ref{lem_descrip_diff_opp}$(i)$.
	The hypothesis $\,Q={\rm Sol}(\mathcal{E})\,$ implies 
	\begin{equation}
	\label{eq_functionals_sols}
	I/\MM^m\,\, = \,\,\bigl\{ \,w \in \FF[\mathbf{y}]/ \MM^m
	\,:\, \delta(w)=0 \,\text{ for all } \,\delta \in \mathcal{E}^\prime\, \bigr\}.
	\end{equation}
	Dualizing the inclusion $\mathcal{E}^\prime \hookrightarrow
	\Hom_\FF\left( \FF[\mathbf{y}]/\MM^m,\FF\right)$ we get the short exact sequence 
	\begin{equation} \label{eq_dualize_short_E}
	0 \,\,\rightarrow \,\,Z \,\,\rightarrow\,\, \FF[\mathbf{y}] / \MM^m
	\,\, \rightarrow \,\,\Hom_\FF(\mathcal{E}^\prime, \FF)\,\, \rightarrow \,\, 0,
	\end{equation}
	where $Z=\Big\lbrace w \in \frac{\FF[\mathbf{y}]}{\MM^m} 
	\,:\, \delta(w)=0 \text{ for all } \delta \in \mathcal{E}^\prime\Big\rbrace$.
	Therefore,  (\ref{eq_functionals_sols}) and (\ref{eq_dualize_short_E}) yield the isomorphism 
	$\,\Hom_\FF(\mathcal{E}^\prime,\FF) \cong \FF[\mathbf{y}]/I$, and 
	we conclude that $\mathcal{E} \cong \mathcal{E}^\prime$ is an $R$-bimodule.
\end{proof}

To complete the proof of Theorem~\ref{thm:main}, it will suffice to prove the following.

\begin{theorem}
	\label{thm_noeth_ops_zero_dim}
	Let $P$ be a maximal ideal in $R = \KK[x_1,\ldots,x_n]$, and 
	let $Q \subset R$ be a $P$-primary ideal 
	of multiplicity $m$ over $P$. Then
	$\,Q = {\rm Sol}(\mathcal{E})$, where
	$\mathcal{E}$ is obtained by the following steps:
	\begin{enumerate}[(i)]
		\item As in Theorem~\ref{thm:param_primary}, set $\,I=\langle y_1,\dots,y_n\rangle^m+\gamma(
		Q)\FF[y_1,\dots,y_n]$.
		\item As in Theorem~\ref{thm:Macaulay_dual}, set $\,V = I^\perp \,\subset\, \FF[z_1,\ldots,z_n]$.
		\item Using the map $\omega$ in (\ref{eq_map_omega}), set  
		$\,\,\mathcal{E}=\omega(V) \subset\FF \otimes_R R\langle\partial_{x_1},\ldots,\partial_{x_n}\rangle\cong\Diff_{R/\KK}(R,\FF)$.		
	\end{enumerate}
\end{theorem} 

\begin{proof}We claim that the correspondence in Proposition~\ref{lem_descrip_diff_opp}$(ii)$ yields
	$$ 		
	\mathcal{E} \,\,\cong\, \,\Hom_{\FF}\bigl( \FF[\mathbf{y}] / I, \FF \bigr)
	\,\, \,\hookrightarrow\, \,\,\Diff_{R/\KK}^{m-1}(R,\FF).
	$$
	The isomorphism (\ref{eq_isom_inv_sys_char_zero}) implies that $V\cong V^\prime = \left(0:_{\FF\left[\mathbf{y^{-1}}\right]} I\right)$.
	Since $I \supseteq \MM^m$, it follows that $V^\prime \subseteq \left(0:_{\FF\left[\mathbf{y^{-1}}\right]} \MM^m\right)$.
	For each $0 \le j < m$, there is a perfect pairing 
	\begin{equation}
	\label{eq_perf_pairing}
\!\!\!\!	{\left[\frac{\FF[\mathbf{y}]}{\MM^m}\right]}_j \!\otimes_\FF {\left[\left(0:_{\FF\left[\mathbf{y^{-1}}\right]} \MM^m\right)\right]}_{-j} \rightarrow \;\FF, \, \mathbf{y}^\alpha \otimes_\FF \frac{1}{\mathbf{y}^\beta} \,\mapsto\, \mathbf{y}^\alpha\cdot\frac{1}{\mathbf{y}^\beta} = \begin{cases}
	1 \;\;\text{ if } \alpha = \beta\\
	0 \;\;\text{ otherwise},
	\end{cases}
	\end{equation}
	where $\vert\alpha\rvert=\vert\beta\rvert=j$, induced by the usual multiplication.
	We get the isomorphisms 
	\begin{equation}
	\label{eq_isom_inv_sys_diff_ops}
	\left(0:_{\FF\left[\mathbf{y^{-1}}\right]} \MM^m\right) \;\cong\; \Hom_\FF\left(\frac{\FF[\mathbf{y}]}{\MM^m},\FF\right) \;\cong\; \Diff_{R/\KK}^{m-1}(R,\FF) .
	\end{equation}
	The second isomorphism is from Proposition~\ref{lem_descrip_diff_opp}$(i)$.
	The Hom-tensor adjunction~gives 
	\begin{multline}\label{eq_isom_V_prime_dual_quot_I}
	V^\prime = \left(0 :_{\left(0:_{\FF\left[\mathbf{y^{-1}}\right]} \MM^m\right)} I\right)
	\cong \Hom_{\FF[[\mathbf{y}]]}\left(\frac{\FF[\mathbf{y}]}{I},
	\Hom_\FF \left( \frac{\FF[\mathbf{y}]}{\MM^m},\FF\right) \right) \\
	\cong \Hom_{\FF}\left(\frac{\FF[\mathbf{y}]}{I}, \FF\right).
	\end{multline}
	The isomorphism $V^\prime \cong \Hom_\FF\left(\frac{\FF[\mathbf{y}]}{I},\FF\right)$ also 
	follows from \cite[Proposition 21.4]{EISEN_COMM}.
	
	By the isomorphism (\ref{eq_isom_inv_sys_char_zero}) and the map $\omega$ in (\ref{eq_map_omega}), $\mathcal{E}$ 
	arises   from $V^\prime$ via the map 
	$$
	V^\prime \xrightarrow{\cong} \mathcal{E}, \quad
	\frac{1}{\mathbf{y}^\alpha} \mapsto \frac{1}{\alpha!}\partial_{\mathbf{x}}^\alpha.
	$$
	On the other hand, by (\ref{eq_diff_opp_z_alpha}), (\ref{eq_perf_pairing}) and (\ref{eq_isom_inv_sys_diff_ops}), the dual monomial $${(\mathbf{y}^\alpha)}^* \in \Hom_\FF\left(\frac{\FF[y_1,\ldots,y_n]}{\MM^m},\FF\right)$$ is identified with the inverted monomial $\frac{1}{\mathbf{y}^\alpha} \in \FF[\mathbf{y^{-1}}]$ and with the differential operator $\overline{D_\mathbf{x}^\alpha} = \frac{1}{\alpha!}\overline{\partial_{\mathbf{x}}^\alpha}\in \Diff_{R/\KK}^{m-1}(R,\FF)$.
	Therefore, the isomorphisms in (\ref{eq_isom_V_prime_dual_quot_I}) imply that $\mathcal{E}$ is indeed determined by $I$ via the correspondence in Proposition~\ref{lem_descrip_diff_opp}$(ii)$. 
	
	After this identification, Proposition~\ref{lem_descrip_diff_opp}$(iii)$ and Theorem~\ref{thm:param_primary} imply that 
	$$
	{\rm Sol}(\mathcal{E})\, =\, \gamma^{-1}(I) \,=\,Q.
	$$
	This completes the proof of Theorem \ref{thm_noeth_ops_zero_dim},
	and we get Theorem \ref{thm:main} for $P$ maximal.
\end{proof}

\subsection{The general case}

We now complete the proof of Theorem~\ref{thm:main}.
As before, $R=\KK[x_1,\ldots,x_n]$,
${\rm char}(\KK) = 0$, and $P $ is prime of height $c$ in $R$. We use the notation 
from Section~\ref{sec4}, where $S=\KK(x_{c+1},\ldots,x_n)[x_1,\ldots,x_c]$ and $\pp = PS$.
By choosing a maximal independent set and permuting variables, 
we can assume that the field extension $\KK(x_{c+1},\ldots,x_n) \hookrightarrow \FF = \text{Quot}(R/P)$ is algebraic.
The ideal $\pp \subset S $ is maximal and $\FF = S/\pp$.
What follows will allow us to derive Theorem~\ref{thm:main} from Theorems~\ref{thm:param_primary},~\ref{thm:Macaulay_dual} and~\ref{thm_noeth_ops_zero_dim}.

\begin{remark}	\label{rem_clear_fractions_diff_ops}
	By Lemma~\ref{lem_diff_ops_R/J}, any $A^{\prime\prime}\in \Diff_{S/\KK(x_{c+1},\ldots,x_n)}^{m-1}(S,S/\pp)$ can be written as 
	$$	A^{\prime\prime} \,\,\,= \sum_{\substack{\beta \in \NN^c\\ \lvert \beta \rvert \le m-1}}
	\overline{h_\beta}\,\, \overline{\partial_{x_1}^{\beta_1}\cdots\partial_{x_c}^{\beta_c}}  
	\quad \text{ for some }\,\, h_\beta \in S.
	$$
	We choose $h \in \KK[x_{c+1},\ldots,x_n]$ such that $h\cdot h_\beta \in R$ for all $\beta$. Hence, we can consider
	$$
	A^\prime \,\,\, = \sum_{\substack{\beta \in \NN^c\\ \lvert \beta \rvert \le m-1}} \overline{h\cdot h_\beta}\, \overline{\partial_{x_1}^{\beta_1}\cdots\partial_{x_c}^{\beta_c}}  \;\in \; \Diff_{R/\KK}^{m-1}(R,R/P).
	$$
	This differential operator satisfies $\,{\rm Sol}(A^\prime)={\rm Sol}_S(A^{\prime\prime}) \cap R$.
\end{remark}

\begin{remark}
	\label{rem_lift_diff_ops_sol}
	Let $A^\prime = \sum_{\lvert \alpha \rvert \le m-1} \overline{r_\alpha} \overline{\partial_{\mathbf{x}}^\alpha} \,\in\, \Diff_{R/\KK}^{m-1}(R,R/P)$ be a differential operator.
	By Lemma~\ref{lem_diff_ops_R/J}, we can lift this to
	$A = \sum_{\lvert \alpha \rvert \le m-1} 
	r_\alpha \partial_\mathbf{x}^\alpha \,\in\, \Diff_{R/\KK}^{m-1}(R,R)$.
	Then, 
	$$ {\rm Sol}(A^\prime)=\lbrace f \in R : A \bullet f \in P \rbrace . $$
\end{remark}

We next describe the Weyl-Noether module in terms of differential operators.

\begin{remark}
	\label{rem_isom_restrict_Weyl_mod}
	We have the following isomorphisms
	\begin{align*}
	\FF \otimes_R D_{n,c} \,\,=\,\,
	\FF \otimes_R R\langle\partial_{x_1},\ldots,\partial_{x_c} \rangle \,&
	\,\cong\,\, \FF \otimes_S \left( S \otimes_R R\langle \partial_{x_1},\ldots,\partial_{x_c}\rangle\right)\\
	& \,\cong \,\, \FF \otimes_S  S\langle \partial_{x_1},\ldots,\partial_{x_c}\rangle\\
	& \,\cong\,\, \Diff_{S/\KK(x_{c+1},\ldots,x_n)}\left(S,\FF\right).
	\end{align*}
	The last isomorphism follows from (\ref{eq_isom_relWeyl_diff}) by applying it to the polynomial ring 
	$S=\KK(x_{c+1},\ldots,x_n)[x_1,\ldots,x_c]$ and the maximal ideal $\pp  = PS$ in $S$.
\end{remark}

\begin{proof}[Theorem~\ref{thm:main}]
	The correspondences between parts (a), (b) and (c) were established in Theorems~\ref{thm:param_primary} and~\ref{thm:Macaulay_dual}.
	Using Remark~\ref{rem_isom_restrict_Weyl_mod}, 
	we identify the Weyl-Noether module
	$\,\FF \otimes_R D_{n,c} \,$ with $\, \Diff_{S/\KK(x_{c+1},\ldots,x_n)}\left(S,\FF\right)$.
	As in (\ref{eq_map_omega}), we consider the map 
	\begin{align}
	\label{eq_map_omega_S}
	\begin{split}
	\omega_S \,:\, \FF[z_1,\ldots,z_c] \;&\rightarrow\; \FF \,\otimes_R \,  D_{n,c}\;\cong\;\Diff_{S/\KK(x_{c+1},\ldots,x_n)}(S,\FF)\\
	z_1^{\alpha_1}\cdots z_c^{\alpha_c} \;&\mapsto\; \partial_{x_1}^{\alpha_1}\cdots\partial_{x_c}^{\alpha_c},
	\end{split}
	\end{align}
	but now applied to the polynomial ring $S=\KK(x_{c+1},\ldots,x_n)[x_1,\ldots,x_c]$ and 
	its maximal ideal $\pp \subset S$.
	This map $\omega_S$ yields the correspondence between parts (c) and (d), that is, between $m$-dimensional $\FF$-vector subspaces of $\FF[z_1,\ldots,z_c]$ that are closed under differentiation and $m$-dimensional $\FF$-vector subspaces of $\FF \otimes_R D_{n,c}$ that are $R$-bimodules
	under the action (\ref{eq_deriv_z_bracket_partial}).
	
	It remains to show that a basis of an $\FF$-vector subspace in part (d) can be lifted to a set of Noetherian operators for the $P$-primary ideal in part (a). 
	For that, let $Q$ be a $P$-primary ideal with multiplicity $m$ over $P$, and set $I=\gamma(Q)$, $V = I^\perp$ and $\mathcal{E} = \omega_S(V)$, by using Theorem~\ref{thm:param_primary}, Theorem~\ref{thm:Macaulay_dual} and (\ref{eq_map_omega_S}), respectively.
	Then, Theorem~\ref{thm_noeth_ops_zero_dim} implies that, for any basis $A_1^{\prime\prime},\ldots,A_m^{\prime\prime}$ of the $\FF$-vector subspace $\mathcal{E} \subset \Diff_{S/\KK(x_{c+1},\ldots,x_n)}(S,\FF)$, we get the equality
	$
	QS = {\rm Sol}_S(A_1^{\prime\prime},\ldots,A_m^{\prime\prime}).
	$
	From Remark~\ref{rem_clear_fractions_diff_ops}, we can choose differential operators 
	$$
	A_i^\prime \,\,=\,\, \sum_{\alpha \in \NN^c} \overline{r_{i,\alpha}} \,\,
	\overline{\partial_{x_1}^{\alpha_1}\cdots\partial_{x_c}^{\alpha_c}} 
	\,\in\, \Diff_{R/\KK}(R,R/P), \quad \text{ where } 1\le i \le m \text{ and } r_{i,\alpha} \in R,
	$$
	such that $Q={\rm Sol}(A_1^\prime,\ldots,A_m^\prime)$.
	By Remark~\ref{rem_lift_diff_ops_sol}, the lifted differential operators
	$$
	A_i \,\,=\,\, \sum_{\alpha \in \NN^c} r_{i,\alpha} \partial_{x_1}^{\alpha_1}\cdots\partial_{x_c}^{\alpha_c} \,\in\, D_{n,c}
	% , \quad \text{ where } 1\le i \le m \text{ and } r_{i,\alpha} \in R,
	$$
	are Noetherian operators for $Q$, hence
	(\ref{eq:fromAtoQ}) holds.
	This completes the proof.
\end{proof}

\section{Symbolic Powers and other Joins}
\label{sec7}

The symbolic power of an ideal is a fundamental construction in commutative algebra.
We here work in the polynomial ring $R=\KK[x_1,\ldots,x_n]$
over a field $\KK$ of characteristic zero, with homogeneous maximal ideal
$\mathfrak{m} = \langle x_1,\ldots,x_n \rangle$. 
The $r$-th {\em symbolic power} of an ideal $J$ in $R$~is
$$ J^{(r)} \,\,\,\,:= \,\,\bigcap_{\pp \in {\rm Ass}(J)} \!\! J^r R_\pp \cap R. $$
Hence, if $P$ is a prime ideal in $R$ then $P^{(r)}$ is the
$P$-primary component of the usual power~$P^r$.
If $\,{\rm codim}(P) = c\,$ then the primary ideal $P^{(r)}$ has multiplicity 
$m =  \binom{c+r-1}{c}$ over $P$, and in Theorem \ref{thm:param_primary}
it is represented by
the zero-dimensional ideal $\,I = \langle y_1,\ldots,y_c \rangle^r\, \subset \, \FF[y_1,\ldots,y_c]$.

Our point of departure  is Sullivant's formula in
\cite[Proposition~2.8]{SULLIVANT_SYMB}:
\begin{equation}
\label{eq:sulli1}
J^{(r)} \,\, = \,\, J \star \mathfrak{m}^r . 
\end{equation}
Here, $J$ is any radical ideal in $R$, and $\star$ denotes the join of ideals.
This is a reformulation of the {\em Zariski-Nagata Theorem}  which
expresses the symbolic power via differential equations:
\begin{equation}
\label{eq:sulli2}
J^{(r)} \,\, = \,\, \,
\biggl\{ \,f \,\in\, R \,\,\,\bigg\vert \,\,\,
\frac{\partial^{i_1+i_2+\cdots+i_n} f}{
	\partial x_1^{i_1}
	\partial x_2^{i_2} \cdots
	\partial x_n^{i_n} 
} \,\in\, J \quad \text{whenever}\, \,\,i_1+i_2 + \cdots + i_n < r \, \biggr\}. 
\end{equation}
The goal of this section is to generalize the equivalence between
(\ref{eq:sulli1}) and (\ref{eq:sulli2}).
We construct $P$-primary ideals by means of joins
and relate this to our earlier results.

\begin{definition}
	If $J$ and $K$ are ideals in $R$, then their \textit{join} is the new ideal
	$$
	J \star K \,\,\,:=\,\,\, \Big( J(\mathbf{v}) \,+\, K(\mathbf{w}) \,+\,
	\langle x_i - v_i - w_i : 1 \le i \le n \rangle \Big) \,\,\cap \,\,R,
	$$
	where $J(\mathbf{v})$ is the ideal $J$ with new variables $v_i$ substituted for $x_i$ and $K(\mathbf{w})$ is the ideal $K$ with  $w_i$ substituted for $x_i$.	
	The parenthesized ideal lives in a polynomial ring in $3n$ variables.
\end{definition}

\begin{remark}
	\label{rem_kernel_map_join}
	Following Simis and Ulrich \cite{SIMIS_ULRICH_JOIN},
	the join $J \star K$ is the kernel of the map 
	\begin{align*}
	R \;&\,\rightarrow \,\;  \frac{\KK[v_1,\ldots,v_n,w_1,\ldots,w_n]}{J(\mathbf{v})+K(\mathbf{w})} \,\;\;\xleftrightarrow{\cong}\; R/J \otimes_\KK R/K\\
	x_i \;&\, \mapsto \,\,\;\; \overline{v_i}+\overline{w_i}  \qquad\qquad\qquad\qquad\leftrightarrow \;\;\overline{x_i} \otimes_\KK 1 + 1 \otimes_\KK \overline{x_i}.
	\end{align*}
	Hence, the quotient $\,R/\left(J\star K\right)\,$ can be identified with a subring of 
	$\,R/J \otimes_\KK R/K$.
\end{remark}

The following result summarizes a few basic properties of the join construction. 

\begin{proposition}
	\label{prop_properties_join}
	Let $J$ and $K$ be ideals in $R$.
	Then, the following statements hold:
	\begin{enumerate}[(i)]
		\item If $J = J_1 \cap J_2$, where $J_1,J_2 \subset R$ are ideals, then $J \star K = (J_1 \star K) \cap (J_2 \star K)$. 
		\item $\sqrt{J \star K} = \sqrt{J} \star \sqrt{K}$; in particular, $J \star K$ is radical when $J$ and $K$ are.
		\item Suppose that $\KK$ is algebraically closed. 
		If $P_1$ and $P_2$ are prime ideals, then $P_1 \star P_2$ is a prime ideal.
		If $J$ and $K$ are primary ideals, then $J \star K$ is a primary ideal. 
		\item If $M$ is an $\mathfrak{m}$-primary ideal, then $P \star M$ is a $P$-primary ideal. 
	\end{enumerate}
\end{proposition}

\begin{proof}
	This is an adaptation of \cite[Proposition 1.2]{SIMIS_ULRICH_JOIN} for non-necessarily homogeneous ideals.
	
	$(i)$  The join distributes over intersections by \cite[Lemma 2.6]{SULLIVANT_SYMB}.
	
	$(ii)$ The ring $R/\sqrt{J} \otimes_\KK R/\sqrt{K}$ is reduced by
	\cite[Corollary 5.57]{GORTZ_WEDHORN}.
	As the kernel of the map 
	$R/J\otimes_\KK R/K \twoheadrightarrow R/\sqrt{J} \otimes_\KK R/\sqrt{K}$ 
	is nilpotent, the claim follows from Remark~\ref{rem_kernel_map_join}.
	
	$(iii)$ Since $\KK$ is algebraically closed, $R/P_1 \otimes_\KK R/P_2$ 
	is an integral domain  \cite[Lemma 4.23]{GORTZ_WEDHORN}.
	By Remark~\ref{rem_kernel_map_join}, $R/(P_1 \star P_2)$ is a subring of
	this domain. Thus, $P_1 \star P_2$ is a prime ideal.
	Suppose ${\rm Ass}(R/J)=\{P_1\}$ and ${\rm Ass}(R/K)=\{P_2\}$. 
	From \cite[Theorem 23.2]{MATSUMURA} we infer
	\begin{equation} \label{eq_equality_ass_primes}
	{\rm Ass}(R/J \otimes_\KK R/K)\,\, = \,\, {\rm Ass}(R/P_1 \otimes_\KK R/P_2).
	\end{equation}
	We already saw that $R/P_1 \otimes_\KK R/P_2$ is an integral domain.
	Therefore, $R/J \otimes_\KK R/K$ has only one associated prime, and 
	hence so does its subring $R/(J \star K)$.
	
	$(iv)$ The equality in (\ref{eq_equality_ass_primes}) is valid for any field. 
	This implies ${\rm Ass}(R/P \otimes_\KK R/M)
	= {\rm Ass}(R/P \otimes_\KK R/\mathfrak{m}) = \{ P \star \mathfrak{m} \} = \{P\}$.
	We hence conclude $\,{\rm Ass}(R / (P \star M))= \{P\}$.
\end{proof} 

\begin{example} In Proposition \ref{prop_properties_join} (iii) we need
	the hypothesis that $\KK$ is algebraically closed.
	If $\KK = \RR$ then 
	$P_1 = \langle x_1^2+1, x_2 \rangle$ and 
	$P_2 = \langle x_1 , x_2^2+1 \rangle $ are prime but their join is not primary:
	$$ P_1 \star P_2 \,\, = \,\, \langle x_1^2+1, x_2^2+1 \rangle \,\, =\,\,
	\langle x_1 - x_2 , x_2^2+1 \rangle \,\, \cap \,\, \langle x_1 + x_2 , x_2^2+1 \rangle .$$
\end{example}

We now focus on the $P$-primary ideals $Q = P \star M$ 
in Proposition \ref{prop_properties_join} (iv).
These will be characterized by differential equations derived from
the $\mathfrak{m}$-primary ideal~$M$.

\begin{definition}
	Fix an $\mathfrak{m}$-primary ideal $M$.
	We encode $M$ by a system $\mathfrak{A}(M)$ of linear PDE with constant coefficients.
	This is done by performing the following~steps:
	\begin{enumerate}[(i)]
		\item Interpret $M$ as PDE by replacing the variables $x_i$ with $\partial_{z_i}$
		for $i=1,\ldots,n$. 
		\item Compute $M^\perp=\left\lbrace F\in \KK[z_1,\ldots,z_n]: f\bullet F=0 \mbox{ for all }f\in M\right\rbrace$. 
		\item Let $\mathfrak{A}(M) \subset \KK[\partial_{x_1},\ldots,\partial_{x_n}]$ be the image of $M^\perp$ under the map $\mathbf{z}^\alpha \mapsto \partial_{\mathbf{x}}^\alpha$.
	\end{enumerate}
	We say the $\KK$-subspace $\mathfrak{A}(M)$ 
	comprises the \textit{differential operators associated to~$M$}.
\end{definition} 

\begin{remark}	\label{rem_joins_props}
	(i) The space $\mathfrak{A}(M)$ is closed under 
	brackets as in (\ref{eq_deriv_z_bracket_partial}) and Theorem~\ref{thm:Macaulay_dual}.
	
	\noindent	(ii) For any $r \ge 1$, we have
	$ \mathfrak{A}\left(\mathfrak{m}^r\right) = \bigoplus_{\lvert \alpha \rvert \le r-1} \KK \, \partial_{\mathbf{x}}^\alpha $. Thus,	 $\mathfrak{A}(\mathfrak{m}^r)$ comprises
	the differential operators  used in the Zariski-Nagata formula for symbolic powers;
	see (\ref{eq:sulli2}) and  \cite[\S 3.9]{EISEN_COMM}.
\end{remark}

The next result generalizes the classical Zariski-Nagata Theorem to 
ideals obtained with the join construction.
Of main interest is the case when  $J=P$ is~prime.

\begin{theorem} \label{thm:ourZN}
	Let $J$ be any ideal in $R = \KK[x_1,\ldots,x_n]$ 
	and let $M$ be $\mathfrak{m}$-primary.
	\begin{enumerate}[(i)]
		\item The join of $J$ and $M$ equals
		$J \star M \;=\;  \big\lbrace f \in R : A \bullet f \in J \;
		\text{ for all }\; A \in \mathfrak{A}(M) \big\rbrace$. 
		\item If $J$ is radical and $r \in \NN$ then $
		J^{(r)} = J \star \mathfrak{m}^r =
		\big\lbrace f \in R : \partial_{\mathbf{x}}^\alpha \bullet f \in J 
		\text{ for  } \lvert \alpha \rvert \le r-1 \big\rbrace $.
	\end{enumerate}
\end{theorem}

\begin{example}
	Let $n=4,c=2$, fix the prime ideal $P$ in (\ref{eq:twistedcubic1}), and consider
	the $\mathfrak{m}$-primary ideal $ M = \langle x_1^2,x_2^2,x_3^2,x_4^2 \rangle$.
	The join $Q = P \star M$ is a $P$-primary ideal of multiplicity $m=11$.
	It is minimally generated by eight octics such as
	$\,x_1^8-4 x_1^6 x_2 x_3+6 x_1^4 x_2^2 x_3^2-4x_1^2 x_2^3 x_3^3+x_2^4 x_3^4$.
	The differential equations from $\mathfrak{A}(M)$
	are simply the squarefree partial derivatives, so that
	\begin{equation}
	\label{eq:repQ1}
	Q \,\,\,= \,\,\,\biggl\{ \,f \,\in\, R \,\,\,\bigg\vert \,\,\,
	\frac{\partial^{i_1+i_2+i_3+i_4} f}{
		\partial x_1^{i_1}
		\partial x_2^{i_2}
		\partial x_3^{i_3}
		\partial x_4^{i_4} 
	} \,\in\, P \quad \text{whenever}\, \,\,i_1,i_2,i_3,i_4 \in \{0,1\}\, \biggr\}. 
	\end{equation}
	This should be compared to the representation by Noetherian operators
	found in Algorithm~\ref{alg:forward}.
	In Step 1, we obtain the ideal
	$I = \langle y_1^4, u_2  y_1^3 y_2 - u_3  y_1 y_2^3, 3 u_1  y_1^2 y_2^2 - 5 u_3  y_1 y_2^3, y_2^4 \rangle$.
	The inverse system $I^\perp$ in Step~2 is the $11$-dimensional
	subspace of $\FF[y_1,y_2]$ spanned by the ten monomials 
	$z_1^{j_1} z_2^{j_2} $ of degree $j_1+j_2\leq 3$ together with
	$$ B({\bf u},{\bf z}) \,\,\,= \,\,\, 2 u_1 u_3 \,z_1^3 z_2\, +\, 5 u_2 u_3 \,z_1^2 z_2^2 
	\,+\, 2 u_1 u_2 \,z_1 z_2^3 .$$
	From Steps 3 and 4 we  obtain
	$$ A({\bf x},\partial_{\bf x}) \,\, = \,\,
	2 x_1 x_3 \partial_{x_1}^3 \partial_{x_2 }
	+ 5 x_2 x_3   \partial_{x_1}^2 \partial_{x_2}^2 
	+ 2 x_1 x_2 \partial_{x_1} \partial_{x_2}^3  . $$
	This gives the following alternative representation of $Q$ by differential equations:
	\begin{equation}
	\label{eq:repQ2}
	Q \,\,\,= \,\,\,\biggl\{ \,f \,\in\, R \,\,\,\bigg\vert \,\,\,
	A \bullet f \in P \,\,\, \,{\rm and} \,\,\,
	\frac{\partial^{j_1+j_2} f}{
		\partial x_1^{j_1}
		\partial x_2^{j_2}
	} \,\in\, P \quad \text{whenever}\, \,j_1 + j_2 \leq 3
	\, \biggr\}. 
	\end{equation}
	The two representations (\ref{eq:repQ1}) and (\ref{eq:repQ2}) differ in
	two fundamental ways. The operators in (\ref{eq:repQ1}) have
	constant coefficients but differentiation involves all four variables.
	In (\ref{eq:repQ2}) we are using an operator from $D_{4,2}$ with polynomial coefficients
	but we differentiate only two variables.
\end{example}

We next show that not every primary ideal arises from the join construction. 

\begin{example}[Palamodov's example]
	\label{exam:Palamodov}
	Let $n = 3$ and $c = 2$, and consider the primary ideal $Q = \langle x_1^2, x_2^2, x_1 - x_2x_3 \rangle$ with $P = \sqrt{Q}=\langle x_1, x_2 \rangle$.
	From \cite[Proposition 4.8 and Example 4.9, page 352]{BJORK} we know that $Q$ cannot be
	described by differential operators with constant coefficients only.
	Theorem~\ref{thm:ourZN} (i) implies that $Q$ does not arise from
	the join construction, i.e.~we cannot find an $\mathfrak{m}$-primary ideal $M$ such that $Q = P \star M$.
	On the other hand, Algorithm~\ref{alg:forward} applied to $Q$ gives the two Noetherian operators $A_1 = 1, A_2 = x_3\partial_{x_1} + \partial_{x_2}$.	
\end{example}

\begin{proof}[Theorem~\ref{thm:ourZN}]
	$(i)$ We use the notation and results from  Section~\ref{sec5}.
	Fix
	an integer $m$ such that $\mathfrak{m}^m \subseteq M$.
	In (\ref{eq_isoms_Diff_R/J})
	we obtained the explicit isomorphism 
	\begin{equation}
	\label{eq_join_isom_diff_ops}
	\Hom_{R/J}\left(R/J \otimes_R P_{R/\KK}^{m-1}, R/J\right) \xrightarrow{\cong} \Diff_{R/\KK}^{m-1}(R,R/J), \quad \psi \mapsto \psi \circ h_{m-1} \circ d^{m-1}
	\end{equation}
	where $h_{m-1}$ is the canonical map $ P_{R/\KK}^{m-1} \rightarrow R/J \otimes_R P_{R/\KK}^{m-1}$ and $d^{m-1}$ is the map in (\ref{eq_univ_diff}).
	Setting $\,T=R \otimes_\KK R =\KK[x_1, \ldots, x_n, y_1, \ldots, y_n]\,$
	as in Section~\ref{sec5}, we have the following isomorphisms:
	\begin{equation}		\label{eq_isom_tensor_prods_join}
	R/J \otimes_R P_{R/\KK}^{m-1}  \;\cong\; \frac{T}{J(\mathbf{x}) \,+
		\, \mathfrak{m}^m(\mathbf{y})} \;\cong\; R/J \otimes_\KK R/\mathfrak{m}^m.
	\end{equation}
	This $\KK$-vector space is considered as an $R$-module 
	via the left factor $R/J \otimes_\KK 1$.	
	By (\ref{eq_join_isom_diff_ops}) and (\ref{eq_isom_tensor_prods_join}), the surjection 
	$R/J \otimes_\KK R/\mathfrak{m}^m \twoheadrightarrow R/J \otimes_\KK R/M$
	gives the~inclusion 
	\begin{equation}
	\label{eq_inclusion_join_diffs}
	\Hom_{R/J}\left(R/J \otimes_\KK R/M, R/J\right) 
	\,\,\hookrightarrow\,\, \Diff_{R/\KK}^{m-1}(R,R/J).
	\end{equation}
	Since  $ R/J \otimes_\KK R/M$ is a finitely generated free $R/J$-module, we have
	\begin{equation*}
	\big\lbrace w \in  R/J \otimes_\KK R/M : \psi(w)=0 \;\text{ for all }\; \psi \in \Hom_{R/J}\left(R/J \otimes_\KK R/M, R/J\right) \big\rbrace \;=\;  \{0\}.
	\end{equation*}
	Let $\,\mathcal{E} \subseteq \Diff_{R/\KK}^{m-1}(R,R/J)\,$ denote the image of
	(\ref{eq_inclusion_join_diffs}). 
	The isomorphism
	(\ref{eq_join_isom_diff_ops}) implies
	$$
	{\rm Sol}(\mathcal{E}) = {\rm Ker}\left(\overline{d^{m-1}}\right), \;\text{ where }\; \overline{d^{m-1}} : R \rightarrow R/J \otimes_\KK R/M, \;\; x_i \mapsto \overline{x_i} \otimes_\KK 1 + 1 \otimes_\KK \overline{x_i}.
	$$
	Therefore, Remark~\ref{rem_kernel_map_join} yields that ${\rm Sol}(\mathcal{E})=J \star M$.
	
	By \cite[Theorem 7.11]{MATSUMURA}, the
	inclusion (\ref{eq_inclusion_join_diffs}) can be written equivalently as
	\begin{align*}
	& \Hom_{R/J}\left(R/J \otimes_\KK R/M, R/J\right)
	\,\, \cong \,\, R/J \otimes_\KK \Hom_\KK(R/M,\KK)\\ & \quad
	\hookrightarrow\; R/J \otimes_\KK \Hom_\KK(R/\mathfrak{m}^m,\KK) \,\, \cong \,\, \Diff_{R/\KK}^{m-1}(R,R/J).
	\end{align*}
	The Hom-tensor adjunction and the perfect 
	pairing in (\ref{eq_perf_pairing}) give the isomorphisms
	$$
	\Hom_\KK\left(R/M,\KK\right) \,\,\cong\,\, \Hom_R\left(R/M, \Hom_{\KK}\left(R/\mathfrak{m}^m,\KK\right)\right) \,\,\cong\,\, \left(0 :_{\KK[\mathbf{x^{-1}}]} M\right).
	$$
	Then, by arguments almost verbatim to those used in the proof of Theorem~\ref{thm_noeth_ops_zero_dim}, we find that 
	$\,\mathcal{E} \subseteq \Diff_{R/\KK}^{m-1}(R, R/J)\,$ is a finitely generated free $R/J$-module,
	and it is generated by 
	$\,\bigl\{ \,\overline{A} \,:\,  A \in \mathfrak{A}(M) \subset \KK[\partial_{\mathbf{x}}] \cap \Diff_{R/\KK}^{m-1}(R,R)\bigr\} \subset \Diff_{R/\KK}^{m-1}(R,R/J)$.
	Summing up, we~conclude 
	$$ J \star M \,\,=\,\, {\rm Sol}(\mathcal{E}) \,\,=\,\,
	\big\lbrace f \in R : A \bullet f \in J \;\text{ for all }\; A \in \mathfrak{A}(M) \big\rbrace.
	$$
	
	$(ii)$  Since $J$ is radical, $J=P_1 \cap \cdots \cap P_k$ for some prime ideals $P_j \subset R$, and so we have $J^{(r)}=P_1^{(r)} \cap \cdots \cap P_k^{(r)}$.
	Proposition~\ref{prop_properties_join}$(i)$ implies $J \star \mathfrak{m}^r = \left(P_1 \star \mathfrak{m}^r\right) \cap \cdots \cap \left(P_k \star \mathfrak{m}^r \right)$.
	Therefore, to finish the proof, it suffices to consider the case where $J=P$ is a prime ideal.
	The Zariski-Nagata Theorem implies  $\,
	P^{(r)} = \big\lbrace f \in R : \partial_{\mathbf{x}}^\alpha \bullet f \in J \; \text{ for all }\; \lvert \alpha \rvert \le r-1 \big\rbrace $.
	The conclusion now follows from part $(i)$ applied to $M=\mathfrak{m}^r$.
	This establishes Theorem \ref{thm:ourZN}.
\end{proof}

\section{Algorithms and their Implementation}
\label{sec8}
 
 In this last section we present practical methods for transitioning between the different
 representations of a $P$-primary ideal $Q$ seen in Theorem \ref{thm:main}.
Our first algorithm computes the Noetherian operators from generators of $Q$.
This provides a framework for describing the space of solutions to the system of
linear PDE associated with $Q$. Namely, we obtain an integral
representation of the solutions via Noetherian multipliers, as promised by the
Ehrenpreis-Palamodov Theorem~\ref{thm:Palamodov_Ehrenpreis}.
Our second algorithm performs the reverse transition, from the solutions to the PDE.
More precisely, here the input is the
list of Noetherian operators and the output is a list of ideal generators for $Q$. The 
Hilbert scheme in Theorem \ref{thm:main} (b) plays a surprising role:
its appearance is responsible for the speed of our computations.

\smallskip

We implemented our algorithms in {\tt Macaulay2}.
The code is made available at 
\begin{equation}
\label{eq:URL}
\hbox{\href{https://software.mis.mpg.de/PrimaryIdealsandTheirDifferentialEquations/index.html}{https://software.mis.mpg.de}.}
 \end{equation}
We hope to develop this into a {\tt Macaulay2} package.
The material in this section extends both the algebraic theory in    \cite{BRUMFIEL_DIFF_PRIM,NOETH_OPS,OBERST_NOETH_OPS}  and the algorithmic steps in \cite{DAMIANO,STURMFELS_SOLVING}.  
  
\smallskip

We now describe our algorithm for computing a set of Noetherian operators of a given $P$-primary ideal $Q$.
The correctness of this algorithm is a direct byproduct of  Theorem~\ref{thm:main}.
In our presentation and examples, we always assume that  $\{x_{c+1},\ldots,x_n\}$ is a maximal independent set modulo $P$.
An algorithm for computing maximal independent sets is described in \cite{KREDEL_INDEP_SET}.
The command \texttt{independentSets} in \texttt{Macaulay2} can be used for this task.

\smallskip

\begin{algo}[From ideal generators to Noetherian operators] \label{alg:forward} \hfill \\
	{\rm Input:} Generators $p_1,p_2,\ldots,p_r$ of a $P$-primary ideal $Q$ in $R =\KK[x_1,\ldots,x_n]$. \\
	{\rm Output:} Elements $A_1,A_2,\ldots,A_m $ in the relative Weyl algebra $D_{n,c}$ that satisfy
	(\ref{eq:fromAtoQ}).
	\begin{enumerate}[(1)]\rm
		\item Compute polynomials in $\FF[y_1,\ldots,y_c]$ that generate the 
	 ideal $I$ in (\ref{eq:corr12}). 		
		\item  Using linear algebra over $\FF$, compute a basis
		$\{B_1,\ldots,B_m\}$ for the inverse system $I^\perp$. 
				
		\item Lift each $B_i(\mathbf{u},\mathbf{z})$ to obtain the Noetherian multipliers $B_i(\mathbf{x},\mathbf{z})$.
		
		\item  Replace $\mathbf{z}$ by $\partial_\mathbf{x}$ to get the Noetherian operators $A_i(\mathbf{x},\partial_\mathbf{x}) $ in (\ref{eq:resultingDO}).
	\end{enumerate}
\end{algo}

\begin{remark}
The computation of the inverse system $I^\perp$ in step (2)  can be implemented
with {\em Macaulay matrices}.
Details on computing Macaulay inverse systems are found in \cite[Chapter~7]{ELKADI_MOURRAIN}. The same applies to step (3) of Algorithm \ref{alg:backward}.
\end{remark}

 The map $\gamma$ from (\ref{eq_map_gamma}) is 
central to Algorithm \ref{alg:forward}. It is used in (1) for
 obtaining a zero-dimensional ideal in $\FF[y_1,\ldots,y_c]$ from a $P$-primary ideal in $\KK[x_1,...,x_n]$.
 Let $\KK[\mathbf{x}, \mathbf{u}, \mathbf{y}] = \KK[x_1,\ldots,x_n,u_1,\ldots,u_n,y_1,\ldots,y_c]$.
 The map $\gamma$ is encoded in the ideal
\begin{equation}
\label{eq:weencode}
\big\langle\, P(\mathbf{u}), \,x_1-y_1-u_1,\,\ldots\,,\,x_c-y_c-u_c\,,\,\, x_{c+1}-u_{c+1}\,,
\,\ldots\,,\,x_n-u_n \,\bigr\rangle \;\subset \; \KK[\mathbf{x}, \mathbf{u}, \mathbf{y}].
\end{equation}
The ideal $P(\mathbf{u})$ is obtained by the substitution $x_i \mapsto u_i$
in a set of generators for $P$.
Depending on the context, the unknown $u_i$ can be seen either as the residue class of $x_i$ in
 $\FF$ or as a variable in $\KK[\mathbf{x}, \mathbf{u}, \mathbf{y}] $.
The roles of the various unknowns are explained in Section \ref{sec2}.
This technique was used for encoding the differential operators in our running example 
in~(\ref{eq:magic}).

\smallskip

The map $\gamma$ also has a crucial role in the algorithm for the backward process.

% \smallskip

\begin{algo}[From Noetherian operators to primary ideal generators] \label{alg:backward} \hfill \\
	{\rm Input:} Elements $A_1,A_2,\ldots,A_m $ in the relative Weyl algebra $D_{n,c}$ that satisfy
	(\ref{eq:leftequalsright}). \\
	{\rm Output:} 
	Generators $p_1,p_2,\ldots,p_r$ of a $P$-primary ideal $Q$
	that is defined as in (\ref{eq:fromAtoQ}). 
	\begin{enumerate}[(1)]\rm
		\item In each $A_i(\mathbf{x},\partial_\mathbf{x})$  replace $\partial_\mathbf{x}$ by $\mathbf{z}$ to obtain the $m$ Noetherian multipliers $B_i(\mathbf{x},\mathbf{z})$ in (\ref{eq:thisresults}).
		
		\item Replace $\mathbf{x}$ by $\mathbf{u}$ to obtain an $\FF$-basis $\{B_1,\ldots,B_m\}$ for the inverse system $I^\perp$. 
		\item  Using $\,\FF$-linear algebra in $\FF[y_1,\ldots,y_c]$, find generators
		for the  ideal~$I$. 
		
		\item  Lift  generators of  $I$ to $\KK[{\bf u}, \mathbf{y}]$ and then add them 
		 to the ideal in (\ref{eq:weencode}). Eliminate
		$\{y_1,\ldots,y_c,u_1,\ldots,u_n\}$ to obtain generators of an ideal $Q' \subset R$.
		Now,  the desired primary ideal $Q$ can be computed as the $P$-primary component of $Q'$.
		
		\textbf{Remark:} This step accounts for the computation of the preimage $\gamma^{-1}(I)$ in the correspondence of (\ref{eq:corr12}). It is trivial in theory  but computationally quite tricky.
		Note that,~in the map $\gamma:R \rightarrow \FF[\mathbf{y}]$, one has that $\FF[\mathbf{y}]$ is not necessarily an algebra of finite type over the polynomial ring $R$. Thus, the computation of $\gamma^{-1}(I)$ cannot be performed directly with the command \texttt{preimage} in \texttt{Macaulay2}.
		We include some details for this step.  
		
		The map $\gamma : R \rightarrow \FF[\mathbf{y}]$ can be expressed as a composition of the map 
		\begin{equation*}
		\gamma':R \hookrightarrow (R/P)[\mathbf{y}] \subset \FF[\mathbf{y}]\, , \qquad
		\begin{matrix}
		x_i  &\mapsto &  y_i+u_i, & \!\!\!\!\! \mbox{ for }1\leq i\leq c,\\
		x_j & \mapsto  & u_j,& \quad \mbox{ for }c+1\leq j\leq n,
		\end{matrix}
		\end{equation*}
		where $u_i$ denotes the class of $x_i$ in $R/P \subset \FF$, and the localization $(R/P)[\mathbf{y}] \hookrightarrow \FF[\mathbf{y}]$.
		Let $I' \subset (R/P)[\mathbf{y}] \subset \FF[\mathbf{y}]$ be an ideal obtained by lifting generators of $I \subset \FF[\mathbf{y}]$.
Then $Q' = {\gamma'}^{-1}(I')$ is found by eliminating $\{y_1,\ldots,y_c,u_1,\ldots,u_n\}$~from 
		$$
		\big\langle \,x_1-y_1-u_1,\,\ldots\,,\,x_c-y_c-u_c\,,\,\, x_{c+1}-u_{c+1}\,,
		\,\ldots\,,\,x_n-u_n \,, \pi^{-1}(I')\,\bigr\rangle,
		$$
		where $\pi$ denotes the canonical map $\pi : \KK[\mathbf{u},\mathbf{y}] \rightarrow (R/P)[\mathbf{y}] \subset \FF[\mathbf{y}]$.
		Finally, since $\FF = W^{-1}(R/P)$, where $W$ is the multiplicative closed subset $W = \KK[x_{c+1},\ldots,x_n] \backslash\{0\}$, it follows that $Q$ is the $P$-primary component of $Q'$.
	\end{enumerate}
\end{algo}

\smallskip

After these technical points,
here are two comments concerning Algorithm~\ref{alg:backward}.

\begin{remark}
{(i)} Step (4) of Algorithm~\ref{alg:backward} can be performed alternatively by using \cite[Algorithm 7]{LEVANDOVSKYY}. We thank one of the reviewers for pointing this out.\\
	 {(ii)} In the input,
	 we may allow elements from the Weyl algebra
$D = R\langle \partial x_1, \ldots,  \partial x_n \rangle$.	Our code in \url{https://software.mis.mpg.de} works for that case.
In the above version of Algorithm~\ref{alg:backward}, we restricted to the
relative Weyl algebra $D_{n,c}$, as this is how the proofs in Sections~\ref{sec4}, \ref{sec5}, \ref{sec6}
are written. Augmenting these to $D$ will require more work, to be deferred
until a better algorithmic understanding
of (\ref{eq:leftequalsright}) is available.
\end{remark}

Here is an example that illustrates our two algorithms
and our \texttt{Macaulay2} code.

\begin{example}
	In Algorithm \ref{alg:forward}, let
	$n = 4$ and fix the prime $P = \langle x_1,x_2,x_3 \rangle$
	that defines a line in $4$-space $\KK^4$. The following ideal is
	$P$-primary of multiplicity $m = 4$:
	$$ Q \,\,=\,\, \bigl\langle\, x_1^2, \,x_1 x_2,\, x_1 x_3, \,x_1 x_4-x_3^2+x_1, \,
	x_3^2 x_4-x_2^2, \,x_3^2 x_4-x_3^2-x_2 x_3+2 x_1 \,\bigr\rangle .$$
	In Step 1 we replace $x_1,x_2,x_3$
	by $y_1,y_2,y_3$ and  $x_4$ by $u_4$
	to get a zero-dimensional ideal $I$ in $\FF[y_1,y_2,y_3]$,
	where $\FF = \KK(u_4)$.
	Note that $I$ contains $\langle y_1,y_2,y_3 \rangle^4$.
	 To check that $I$ is a point in ${\rm Hilb}^4(\FF[[y_1,y_2,y_3]])$, we exhibit 
 a flat deformation to the square of the maximal ideal:
	$$ I \,\,=\,\,\bigl\langle\, y_1^2\,,\,y_1 y_2\,,\, y_1 y_3\, ,\,
	y_2^2 -(u_4^2+u_4) \,y_1\,,\,
	y_2 y_3 - (u_4^2 + 1) \,y_1\,,\,
	y_3^2 - (u_4+1)\, y_1\,\bigr\rangle. $$
	We refer to \cite[Example 18.9]{MILLER} for more details. The inverse system $I^\perp$ lives in $\FF[z_1,z_2,z_3]$. It is the $4$-dimensional $\FF$-vector space with basis
	$$ B_1 \,=\,  (u_4^2+u_4) z_2^2 + 2 (u_4^2+1)z_2 z_3 + (u_4+1) z_3^2 +2z_1\,,\,
	B_2 =  z_2\,,\, B_3 = z_3\,, \,B_4 = 1. $$
	This space is closed under differentiation. The Noetherian operators in Step 4 are
	$$ A_1 \,=\,
	(x_4^2+x_4) \partial_{x_2}^2 + 2 (x_4^2+1)
	\partial_{x_2} \partial_{x_3} + (x_4+1) \partial_{x_3}^2 +2 \partial_{x_1},\,
	A_2 = \partial_{x_2\,}, \, A_3 = \partial_{x_3} \,, \, A_4 = 1 . $$
	We now check that these four operators in $D_{4,3}$  
	represent the given primary ideal:
	$$ Q  \,\,= \,\,\bigl\{\,f \in \KK[x_1,x_2,x_3,x_4]:\, A_i \bullet f\in\langle x_1,x_2,x_3\rangle 
	\hbox{ for $i=1,2,3,4$} \,\bigr\} .$$
	Reversing  this entire computation is the point of Algorithm \ref{alg:backward}.
	Starting from the operators $A_1,A_2,A_3,A_4$, we compute 
	the polynomials $B_1,B_2,B_3,B_4$ in $\FF[z_1,z_2,z_3]$, which span
	the inverse system $I^\perp$. 
	In Step 3, we find generators of the ideal $I$
	in $\FF[y_1,y_2,y_3$]. And, finally, from this one obtains generators
	of $Q$ by the elimination process described in Step 4.	

	After saving the code posted at  (\ref{eq:URL})
 in a file called \texttt{noetherianOperatorsCode.m2}, we execute the following \texttt{Macaulay2} session.

	\verbatimfont{\scriptsize}%
\begin{verbatim}
i1 : load "noetherianOperatorsCode.m2"
i2 : R = QQ[x1,x2,x3,x4];
i3 : Q = ideal(x1^2,x1*x2,x1*x3,x1*x4-x3^2+x1,x3^2*x4-x2^2,x3^2*x4-x3^2-x2*x3+2*x1);
i4 : L = time getNoetherianOperatorsHilb(Q)
     -- used 0.106942 seconds
                      2         2       2                          2
o4 = {1, dx2, dx3, (x4  + x4)dx2  + (2x4  + 2)dx2*dx3 + (x4 + 1)dx3  + 2dx1}

i5 : Q2 = time getIdealFromNoetherianOperators(L, radical Q)
     -- used 0.115953 seconds
              2                        2                                
o5 = ideal (x3  - x1*x4 - x1, x1*x3, x2  - x2*x3 - x1*x4 + x1, x1*x2, 
       2       2
     x1 , x1*x4  - x2*x3 + x1)

i6 : Q == Q2                  -- check that the two ideals are equal
o6 = true
\end{verbatim}		
\noindent
By Theorem~\ref{thm:Palamodov_Ehrenpreis}, any solution to 
(\ref{eq:mustsolvethis}) of PDE encoded by $Q$ can be represented~as
	$$
	\psi(\mathbf{z}) \,\,= \,\, 
	\int \exp\left(x_4 z_4\right) d\mu_1(x_4) + \int z_2\exp\left(x_4 z_4\right) d\mu_2(x_4) +\int z_3\exp\left(x_4 z_4\right) d\mu_3(x_4)+
	$$
$$+\int \left((x_4^2+x_4)z_2^2+(x_4^2+1)z_2z_3+(x_4+1)z_3^2+2z_1\right)\exp\left(x_4 z_4\right) d\mu_4(x_4), $$

\noindent
where $\mu_1, \dots, \mu_4$ are measures on the line  $\,V(\langle x_1,x_2,x_3 \rangle) = \{(0,0,0,x_4)\}$.
\end{example}

Finally, we apply our Algorithm \ref{alg:forward} to the PDE discussed in Example \ref{ex:september17}.
The computations below show the scope and capability of our algorithm to solve PDE with constant coefficients via the Ehrenpreis-Palamodov Fundamental Principle.

\begin{example} \label{ex:september21}
We now recall the systems  described in Example \ref{ex:september17}. Our goal is to provide a finite representation of its non-degenerate solutions $\psi({\bf z})$ via Noetherian multipliers.
  We begin by illustrating the use of Algorithm \ref{alg:forward}
 for the case $k=3$:
\begin{equation}\label{k=3} \begin{matrix}
\!\! \bigl(\frac{\partial^2}{ \partial z_1^2} - \frac{\partial^2}{ \partial z_2 \partial z_3} \bigr)^{\! 3} \! \psi({\bf z})
\, = \,
\bigl(\frac{\partial^2}{ \partial z_1 \partial z_2} - \frac{\partial^2}{ \partial z_3 \partial z_4} \bigr)^{\! 3} \! \psi({\bf z})
\, = \,
\bigl(\frac{\partial^2}{ \partial z_2^2} - \frac{\partial^2}{ \partial z_1 \partial z_4} \bigr)^{\! 3} \! \psi({\bf z})
\, = \, 0 . \end{matrix}
\end{equation}

\noindent
The first step is to encode (\ref{k=3}) in a polynomial ideal $J$. This is done in  \texttt{Macaulay2}:

\verbatimfont{\scriptsize}
\begin{verbatim}
i1 : load "noetherianOperatorsCode.m2"
i2 : R = QQ[x1,x2,x3,x4];
i3 : k=3;
i4 : J = ideal((x1^2-x2*x3)^k,(x1*x2-x3*x4)^k,(x2^2-x1*x4)^k);
\end{verbatim}

\noindent
The saturation step in (\ref{eq:ourintention}) replaces the three order six operators in (\ref{k=3}) with
a new system of six linear PDE whose solutions are precisely the non-degenerate solutions 
to (\ref{k=3}):

\verbatimfont{\tiny}
\begin{verbatim}
i5 : Q = saturate(J,ideal(x1*x2*x3*x4))
              5       3  2           3          4  2      2        2     2  2  2        2  3       4  
o5 = ideal (x2 x3 + x1 x2 x4 - 4x1*x2 x3*x4 - x1 x4  + 3x1 x2*x3*x4  + x2 x3 x4  - x1*x3 x4 , x1*x2 x3
     --------------------------------------------------------------------------------------------------
         4           2  2          3  2       3     2            2  2     3  3    2  3       4  2  
     + x1 x2*x4 - 3x1 x2 x3*x4 - x2 x3 x4 - x1 x3*x4  + 4x1*x2*x3 x4  - x3 x4 , x1 x2 x3 - x2 x3  +
     --------------------------------------------------------------------------------------------------
       5        3                 2  2       2  2  2        3  2    6         4        2  2  2  
     x1 x4 - 4x1 x2*x3*x4 + 3x1*x2 x3 x4 + x1 x3 x4  - x2*x3 x4 , x2  - 3x1*x2 x4 + 3x1 x2 x4  -
     --------------------------------------------------------------------------------------------------
       3  3    3  3      2  2                 2  2     3  3    6      4           2  2  2     3  3
     x1 x4 , x1 x2  - 3x1 x2 x3*x4 + 3x1*x2*x3 x4  - x3 x4 , x1  - 3x1 x2*x3 + 3x1 x2 x3  - x2 x3 )
\end{verbatim}

\noindent
The system of PDE  we need to solve is  obtained after replacing $x_i$ by $\partial_{z_i}$ in the ideal $Q$ above. 
The new system has six differential equations, the last one being 

$$\left(\frac{\partial^6}{\partial z_1^6}-3\frac{\partial^6}{\partial z_1^4z_2z_3}+3\frac{\partial^6}{\partial z_1^2z_2^2z_3^2}-\frac{\partial^6}{\partial z_2^3 z_3^3}\right)\psi({\bf z})=0.$$

\smallskip

\noindent
Algorithm \ref{alg:forward} returns the Noetherian operators that encode the $P$-primary ideal~$Q$:

\verbatimfont{\scriptsize}
\begin{verbatim}
i6 : getNoetherianOperatorsHilb(Q)
                      2              2           3          2                 2          3
o6 = {1, dx1, dx2, dx1 , dx1*dx2, dx2 , - 4x3*dx1  - 6x1*dx1 dx2 + 6x2*dx1*dx2  + 4x4*dx2 }
\end{verbatim}

We now replace $\partial_{x_i}$ by $z_i$ in the Noetherian operators in {\tt o6}.
This  gives  seven Noetherian multipliers: $B_1 {=}1, B_2 = z_1, B_3 =z_2, B_4=z_1^2, B_5=z_1z_2, B_6=z_2^2 , B_7=-4x_3z_1  - 6x_1z_1 ^2z_2 + 6x_2 z_1 z_2^2  + 4x_4 z_2^3$. Then non-degenerate solutions of (\ref{k=3}) are 
	$$
	\psi(\mathbf{z}) \,\,= \,\, \sum_{i=1}^7
	\int_{V(P)} \!\!\!\! B_i({\bf x},{\bf z})\exp\left(\mathbf{x}^t \mathbf{z}\right) d\mu_i(\mathbf{x}),
	$$

\noindent	
where $\mu_i$ are measures supported on $V(P)$, as in Example \ref{ex:NoetMult}. 
Explicit solutions can be found with the parametrization  given in the Introduction. 
For example, taking $\mu_7$ as the Dirac measure at the point $(2,3)$  and setting 
$\mu_1=\dots=\mu_6=0$, we~get
$$\,\psi(\mathbf{z}) = (-32z_1  - 72z_1 ^2z_2 + 108 z_1 z_2^2  + 108 z_2^3) \,{\rm exp} (12 z_1 + 18 z_2 + 8 z_3 + 27 z_4).$$

We now consider larger values of $k$ in Example \ref{ex:september17}.
Again, our aim is to compute all  non-degenerate solutions of the three
given differential equations of order $2k$. Our implementation of
Algorithm \ref{alg:forward} is quite successful in carrying
out this computation, even as $k$ increases. 
We performed this computation up to $k=12$. The results are reported in
Table \ref{tab:impressive}. The first column gives the number of
minimal generators of the primary ideal $Q$. The second column
shows the length $m$ of $Q$ over $P$. The third column shows
the running time
on a MacBook Pro with a \emph{2,7 GHz Dual-Core Intel Core i5 processor} and \emph{8 GB 1867 MHz DDR3}.

\begin{table}[h]
\begin{center}
\begin{tabular}{|c|ccc|}
\hline
k  & \# PDE & \# Noetherian operators & time (sec)\\
\hline\hline
%1 & 3 & 1 & 0.0560627\\
%\hline
%2 & 6 & 3 & 0.13232\\
%\hline
%3 & 6 & 7 & 0.270294 \\
%\hline
%4 & 10 & 12 & 0.786478\\
%\hline
%5 & 12 & 19 & 4.15943\\
%\hline
%6 & 24 & 27 & 10.6518 \\
%\hline
%7 & 18 & 37 & 45.8213\\
%\hline
%8 & 31 & 48 & 82.6401\\
%\hline
%9 & 25 & 61 & 275.584\\
%\hline
%10 & 54 & 75 & 470.572\\
%\hline
%11 & 39 & 91 & 1290.45\\
%\hline
%12 & 66 & 108 & 2050.76\\
%\hline
1  &  3  &  1  &  0.042  \\
\hline
2  &  6  &  3  &  0.068  \\
\hline
3  &  6  &  7  &  0.216 \\
\hline
4  &  10  &  12  &  0.514 \\
\hline
5  &  12  &  19  &  2.571  \\
\hline
6  &  24  &  27  &  6.306  \\
\hline
7  &  18  &  37  &  28.502 \,  \\
\hline
8  &  31  &  48  &  51.378\,\,  \\
\hline
9  &  25  &  61  &  196.969\,\, \, \\
\hline
10  &  54  &  75  &  308.750\,\, \, \\
\hline
11  &  39  &  91  &  929.884 \,\,\, \\
\hline
12  &  66  &  108 \! &  1359.310 \,\,\, \,\, \\
\hline
\end{tabular}
\end{center}
\caption{\label{tab:impressive} Computation times of Algorithm \ref{alg:forward} 
for Example \ref{ex:september17} with $k\leq 12$.}
\end{table}

The table shows that,
using {\tt Macaulay2}, we are able to
 solve some rather nontrivial PDE with constant coefficients.
 For instance, for $k = 12$, 
the primary ideal $Q$ is generated by $66$ polynomials whose degrees vary between $24$ and $30$.
The reader can view the massive size of this ideal in the posting
at the URL (\ref{eq:URL}).
\end{example}

\begin{remark}
One nice feature of our implementation of Algorithm \ref{alg:forward}
is that the input ideal need not be primary. All we require
is that the radical of the input ideal is a prime ideal $P$.
The zero-dimensional ideal computed in the first step of the {\tt Macaulay2} code only reflects the $P$-primary component of the input. Hence the algorithm finds the Noetherian operators corresponding
to that $P$-primary ideal.
For instance, in Example \ref{ex:september21},
we can replace  {\tt i6} by
\begin{verbatim}
i7 : getNoetherianOperatorsHilb(J)
\end{verbatim}
This produces the same output, namely the seven
Noetherian operators in {\tt o6}.
\end{remark}

\end{document}